\documentclass[a4paper,10pt,reqno]{amsart}

\usepackage{hyperref}
\usepackage{amssymb}
\usepackage{eucal}
\usepackage{eufrak}
\usepackage{graphics}
\usepackage{graphicx}
\usepackage{epsfig}
\usepackage{multicol}
\usepackage{xypic}
\usepackage{amsmath}
\usepackage{wasysym}
\usepackage{dsfont}

\usepackage{eufrak}
\usepackage{graphics}
\usepackage{graphicx}
\usepackage{epsfig}
\usepackage{multicol}
\usepackage{xypic}
\usepackage{amsmath}
\usepackage{color}

\newtheorem{proposition}{Proposition}[section]
\newtheorem{definition}{Definition}[section]
\newtheorem{lemma}[definition]{Lemma}
\newtheorem{remark}[definition]{Remark}

\newtheorem{theorem}[definition]{Theorem}

\newtheorem{corollary}[definition]{Corollary}

\newcommand{\lp}{\left(}
\newcommand{\rp}{\right)}
\newcommand{\lc}{\left\{}
\newcommand{\rc}{\right\}}
\newcommand{\der}{\partial}

\newcommand{\bra}{\langle}
\newcommand{\ket}{\rangle}
\newcommand{\R}{\mathds{R}}      
\newcommand{\N}{\mathds{N}}      
\newcommand{\F}{\mathds{F}}

\newcommand{\Flder}{\rightarrow}

\usepackage{amssymb}

\newcommand{\proa}{A^*G \mbox{$\;$}_{\tau^*} \kern-3pt\times_\alpha
G \mbox{$\;$}_\beta \kern-3pt\times_{\tau^*} A^*G}

\newcommand{\tf}{\tilde f}
\newcommand{\Arg}{(t,x,\dot x)}
\newcommand{\Argt}{(t,x(t),\dot x(t))}
\newcommand{\Argo}{(t_0,x_0,\dot x_0)}

\begin{document}

\title{On the discretization of nonholonomic dynamics in $\R^n$}

\author[F. Jim\'enez]{Fernando Jim\'enez}
\address{F. Jim\'enez: Zentrum Mathematik der Technische Universit\"at M\"unchen, D-85747 Garching bei M\"unchen, Germany} \email{fjimenez@ma.tum.de}

\author[J. Scheurle]{J\"urgen Scheurle}
\address{J. Scheurle: Zentrum Mathematik der Technische Universit\"at M\"unchen, D-85747 Garching bei M\"unchen, Germany} \email{scheurle@ma.tum.de}

\thanks{This research was supported by the DFG Collaborative Research Center TRR 109, ``Discretization in Geometry and Dynamics''. 
}



\maketitle

\begin{abstract}
In this paper we explore the nonholonomic Lagrangian setting of mechanical systems in local coordinates on finite-dimensional configuration manifolds. We prove existence and uniqueness of solutions by reducing the basic equations of motion to a set of ordinary differential equations on the underlying distribution manifold $D$. Moreover, we show that any $D-$preserving discretization may be understood as beeing generated by the exact evolution map of a time-periodic non-autonomous perturbation of the original continuous-time nonholonomic system. 
By means of discretizing the corresponding Lagrange-d'Alembert principle, we construct geometric integrators for the original nonholonomic system. We give precise conditions under which these integrators generate a discrete flow preserving the distribution $D$. Also, we derive corresponding consistency estimates. Finally, we carefully treat the example of the nonholonomic particle, showing how to discretize the equations of motion in a reasonable way, particularly regarding the nonholonomic constraints. 
The exploration in this paper lays the ground to analyze the dynamics of appropriate discretizations of nonholonomic mechanical systems in the Lagrangian framework  and to relate that dynamics to the dynamics of the original nonholonomic systems. We postpone this analysis to a series of forthcoming papers.   
\end{abstract}

\section{Introduction}

Nonholonomic systems play an important role in Mechanics. This kind of systems are characterized by so-called nonholonomic constraints, i.e. cons\-traints involving both configuration as well as velocity variables, and which can not be integrated to {\it purely} configuration-dependent constraints (in this case the constraints are called holonomic). Geometrically, nonholonomic constraints, when they are linear in the velocities, may be understood as a distribution $D$ on the tangent bundle of a given configuration  manifold $Q$, i.e. $D\subset TQ$. The dynamical behavior of nonholonomic systems, for instance related to perfect rolling motion, substantially differs from that of holonomic ones in various aspects. While the energy is conserved in nonholonomic systems, this is not necessarily true for other quantities which are conserved in the holonomic case. Usually there is no invariant Poisson bracket associated with nonholonomic systems. Moreover, volume forms in the Lagrangian phase space are not invariant. The corresponding equations of motion, termed {\it nonholonomic} are derived by the Lagrange-d'Alembert principle. We point out, that different equations, termed {\it vakonomic}, result from a genuine variational principle. These do not describe the motion of a nonholonomically constrained mechanical systems correctly (but others, such as underactuated control systems, etc.).  

Many authors have recently shown new interest in the nonholonomic theory and also in its
relation to recent developments in control theory and robotics. The main innovation is that nonholonomic systems are studied from a geometric perspective; see \cite{Bl,BlKrMaMu1996,CCMD,Cort,Koiller,LD} and references therein. Due to its practical importance, nonholonomic mechanics has also attracted a lot of attention by people from the numerical integration community.

Within the general framework of geometric integration (\cite{Hair}), a primary requirement for a discretized dynamical system is a reasonable correspondence of principal qualitative features of the dynamics, which is partly achieved by respecting geometric structural properties of the original continuous-time system. However, it is in general impossible to achieve a complete qualitative and quantitative correspondence (see \cite{GeMar} for a formal statement in this regard). Thus, any practical integrator is likely to disrespect some of the qualitative features of the original system.

Recent works, such as \cite{CoMa,FeZen,SDD1,SDD2,SJD,IMMM,KFMD,McPer} have introduced numerical integrators for
nonholonomic systems with very good energy behavior and preservation properties such as the
preservation of the discrete nonholonomic momentum map in case of a system with symmetry. The approaches in most of these references are based on the ideas of \cite{MarsdenWest,Mose}, where the continuous Lagrange-d'Alembert principle on $TQ$ is replaced by discrete Lagrange-d'Alembert principles on the discrete phase space $Q\times Q$. Of special interest are the seminal works on nonholonomic integration \cite{CoMa,McPer}, where a discrete version of the Lagrange-d'Alembert principle is proposed by introducing a proper discretization of the nonholonomic distribution, i.e. $D_d\subset Q\times Q$, which provides integrators with important structural preservation properties. Nevertheless, as just mentioned, this nice behaviour does not imply a {\it complete} agreement between continuous and discrete counterparts. Since the principal geometric object describing nonholonomic constraints is the distribution $D$, in this paper we focus on integrators which {\it exactly} respect the original continuous distribution $D$ or a certain perturbation of that. 

It turns out that  this $D-$preservation {\it is not} obvious from the discretization of the Lagrange-d'Alemebert principle.  This raises the question about the inavoidable discrepancies introduced by discretizating the dynamics, even if the latter is performed in some kind of structure-preserving fashion. It is needless to mention that this is a central and fundamental question for all kinds of numerical investigations, especially concerning the long-term evolution of dynamical systems. Regarding this issue, we refer to \cite{FiSch} where a positive answer to the following question is given: Is it possible to embed a numerical scheme approximating the continuous-time flow of a set of autonomous ordinary differential equations (ODE) into the time evolution corresponding to a non-autonomous perturbation of the original autonomous ODE? For convenience, we recall this result in proposition \ref{propoDisc} of the present paper. It  may be phrased as follows: {\it Any $p-$th order discretization of an autonomous ODE can equivalently be viewed as the time$-\epsilon$ period map of a suitable $\epsilon-$periodic non-autonomous perturbation of the original ODE (where $\epsilon$ is the fixed discretization lenght).} 

In the present paper we apply this result to the case of nonholonomic systems in $Q=\R^n$. For that purpose, we first rewrite the basic equations of motion as an appropriate set of ODEs in local coordinates. This turns out to be possible under some regularity conditions on the Lagrangian function and the distribution manifold, which we assume throughout the paper. Thus, we are able to apply the result from \cite{FiSch} to the derived ODE system leading to the general result that any $D-$preserving discretization of an autonomous nonholonomic problem 
may be understood as being generated by the exact evolution map of a time-periodic non-autonomous continuous-time perturbation of the original system. Moreover, by means of discretizing the corresponding Lagrange-d'Alembert principle, cf. \cite{CoMa,McPer}, we construct geometric integrators for the original nonholonomic system. We give precise conditions under which the resulting discretization schemes preserve the distribution $D$. In this case, a geomteric integrator will be called {\it velocity nonholonomic integrator}. We provide two specific examples of such an integrator, apply these to the case of {\it simple mechanical systems}, and study consistency properties with respect to the continuous-time dynamics. Finally, we consider the model of the nonholonomic particle to illustrate the theory.

\medskip

The paper is structured as follows: in $\S$\ref{NonHoloProblem} we introduce the notion of a Lagrangian nonholonomic problem in $\R^n$ with linear velocity constraints and show how it can be undertstood as a set of differential algebraic equations (DAE). For later use, we present a proof for the existence and uniqueness of solutions of a non-autonomous generalization of this system following \cite{RigidBodyDAE}. Moreover, it is proved in proposition \ref{h} that the nonholonomic DAE can be rewritten as an ODE on the constraint manifold $D$  (distribution $D$) 
underlying the nonholonomic setting. $\S$\ref{Interpo} is devoted to the exploration of the possibility of interpolating two points belonging to $D$ by a curve within $D$ itself. This fact is proved in proposition \ref{InterpolationD} taking advantage of coordinates on $D$ suggested by using an appropriate Ehresmann connection to trivialize the corresponding bundle structure. In $\S$\ref{DLNP} we first recall the above mentioned result by Fiedler and Scheurle (proposition \ref{propoDisc}). Then, this result is generalized to the present case, see corollary \ref{Coro}. 
Moreover, we introduce the notion of discretization of a  nonholonomic flow and use this together with the results obtained in $\S$\ref{Interpo} to explore the perturbation of the original nonholonomic system caused by that kind of discretizations; see remark \ref{LargeRemark}. Section \S\ref{section2} is devoted  to the construction of {\it velocity nonholonomic integrators}. This notion is introduced in definition \ref{VNI}. A special type of such integrators is developed by means of the discretization of the Lagrange-d'Alembert principle established in \cite{CoMa,McPer}. 
In proposition \ref{propoDcont} we prove that those inegrators are $D-$preserving in some sense. Using the results of the previous section, in $\S$\ref{examples}, for two specific examples of such integrators, the order of consistency is carefully explored with respect to the class of {\it simple mechanical systems}; see propositions \ref{firstorder} and \ref{secondorder}. Moreover, the case of the {\it nonholonomic particle} is considered as an elucidating toy model. In particular, this system suggests the extension of the theory proposed in $\S$\ref{DLNP}, such that a perturbation of the constraints due to discretization is allowed. Such an extension is finally presented in remark \ref{deformation}.   

Throughout the paper, we use Einstein's convention for the summation over repeated indices. The proofs of some known results which are important for the understanding of our main results are included for convenience of the reader; it will be properly highlighted.

\section{The Lagrangian nonholonomic problem}\label{NonHoloProblem}

Mathematically, the Lagrangian nonholonomic setting can be described as follows. We shall start with a configuration manifold $Q$, which is assumed to be an $n$-dimensional differentiable manifold with local coordinates denoted by $q^{i},\hspace{1mm} i=1,...,n$, and a non-integrable constant-rank distribution $D$ on $Q$ that describes the linear nonholonomic constraints. We can consider this distribution $D$ as a vector subbundle of the tangent bundle $TQ$ (velocity phase space) of the configuration manifold. Locally, the linear constraints are written as follows:
\begin{equation}\label{LC}
\phi^{\alpha}\lp q,\dot q\rp=\mu^{\alpha}_{i}\lp q\rp\dot q^{i}=0,\hspace{2mm} 1\leq \alpha\leq m,
\end{equation}
where $\mbox{rank}\lp D\rp=n-m$. The annihilator $D^{\circ}$ is locally given by
\[
D^{\circ}=\mbox{span}\lc\mu^{\alpha}=\mu_{i}^{\alpha}(q)\,dq^{i};\hspace{1mm} 1\leq \alpha\leq m\rc,
\]
where the one-forms $\mu^{\alpha}$ are supposed to be linearly independent.

In addition to the distribution, we need to specify the dynamical evolution of the system, for example by choosing a Lagrangian function $L:TQ\Flder\R$. In nonholonomic mechanics, the procedure leading from the Newtonian point of view to the Lagrangian one is given by the Lagrange-d'Alembert principle. This principle says that a curve $q:I\subset \R\Flder Q$  is an admissible motion of the system if
\[
\delta\int^{t_1}_{t_0}L\lp q\lp t\rp, \dot q\lp t\rp\rp dt=0
\]
with respect to all variations such that $\delta q\lp t\rp\in D_{q\lp t\rp}$, $t_0\leq t\leq t_1$ and the fixed end point condition is satisfied, and if the velocity of the curve itself satisfies the constraints. It is remarkable that the Lagrange-d'Alembert principle is not variational since we are imposing the constraints on the curve {\it after} the extremization. Thus, one may consider the intrinsic data defining the Lagrangian nonholonomic problem to be given by the triple $(Q,L,D)$. Using Lagrange-d'Alembert's principle, we arrive at the nonholonomic equations, which in coordinates read 
\begin{subequations}\label{LdAeqs}
\begin{align}
\frac{d}{dt}\lp\frac{\der L}{\der\dot q^{i}}\rp-\frac{\der L}{\der q^{i}}&=\lambda_{\alpha}\,\mu^{\alpha}_{i}(q),\label{Con-1}\\\label{Con-2}
\mu_{i}^{\alpha}(q)\,\dot q^{i}&=0,
\end{align}
\end{subequations}
where $\lambda_{\alpha},\hspace{1mm} \alpha=1,...,m$ are ``Lagrange multipliers''. The right-hand side of equation (\ref{Con-1}) represents the reaction forces due to the constraints, and equations (\ref{Con-2}) represent the constraints themselves.


\subsection{Local existence and uniqueness of solutions}
In this work we are mainly concerned with the case $Q=\R^n$, so we will assume in the following that we deal with this particular configuration manifold. However, some of the results can be extended to a general manifold $Q$, a fact that will be pointed out at places in the paper. 

The equations \eqref{LdAeqs}, 
\begin{subequations}\label{NhDAE1}
\begin{align}
\frac{\der^2L}{\der\dot q^i\der\dot q^j}\,\ddot q^j&=\frac{\der L}{\der q^i}-\frac{\der^2L}{\der\dot q^i\der q^j}\,\dot q^j�+\lambda_{\alpha}\,\mu^{\alpha}_i(q),\label{NhDAE1a}\\
\mu_{i}^{\alpha}(q)\,\dot q^{i}&=0,\label{NhDAE1c}
\end{align}
\end{subequations}
together with the initial conditions 
\begin{equation}\label{InitialData}
q(t_0)=q_0\in\R^n,\,\,\,\dot q(t_0)=\dot q_0\in\R^n,
\end{equation}
represent a set of $n+m$ differential-algebraic equations in $\R^{n}\times\R^m$ (called nonholonomic DAEs subsequently). We are interested in the existence and uniqueness of solutions $(q,\lambda):I\subset\R\Flder\R^n\times\R^m$ of \eqref{NhDAE1} and \eqref{InitialData}. In order to address this issue, we refer to the following general result in \cite{RigidBodyDAE}. We slightly modify our notation to state theorem \ref{EUDAETh}. In particular, we denote the configuration variables by $x$ here, in order to stick to the notation in the book 
by Rabier and Rheinboldt \cite{RigidBodyDAE}.

Consider the general non-autonomous DAE
\begin{subequations}\label{NonAutDAE}
\begin{align}
A(x)\ddot x&=F(t,x,\dot x)+\lp\nabla_{\dot x}\varphi(t,x,\dot x)\rp^{T}\lambda,\label{NonAutDAEa}\\
\varphi(t,x,\dot x)&=0,\label{NonAutDAEb}
\end{align}
\end{subequations}
where $\lambda\in\R^m$, $x\in\R^n$, $\nabla_{\dot x}=\frac{\der}{\der\dot x}$ and $(t,x,\dot x)$ denotes an arbitrary element of $\R\times\R^n\times\R^n$, while $(t,x(t),\dot x(t))$ refers to a curve $t\mapsto x(t)$ with $\dot x(t)=(dx/dt)(t)$. The coefficient functions
\begin{equation}\label{Coeff}
\lc
\begin{array}{l}
A:\R^n\Flder L(\R^n,\R^n),\\
F:\R\times\R^n\times\R^n\Flder\R^n,\\
\varphi:\R\times\R^n\times\R^n\Flder\R^m,\,\,m\leq n,
\end{array}
\rc
\end{equation}
are assumed to be defined and smooth in some appropriate open sets. With these ingredients the following general solvavility result can be established
\begin{theorem}\label{EUDAETh}
Suppose that, in \eqref{NonAutDAE}, the coefficient functions \eqref{Coeff} are defined and smooth in some neighborhood of a point $(t_0,x_0,\dot x_0)\in\R\times\R^n\times\R^n$ and that
{\rm
\begin{equation}\label{InitialCoeff}
\lc
\begin{array}{l}
(i)\,\,\,\,\,\,\,\,\,\bra A(x_0)z,z\ket,\,\,\forall\,z\in\mbox{ker}\,\nabla_{\dot x}\varphi(t_0,x_0,\dot x_0),\,\,z\neq 0,\\
(ii)\,\,\,\,\,\,\,\,\varphi(t_0,x_0,\dot x_0)=0,\\
(iii)\,\,\,\,\,\,\mbox{rank}\,\nabla_{\dot x}\varphi(t_0,x_0,\dot x_0)=m.
\end{array}
\rc
\end{equation}
}
Then there exists a $\delta>0$ such that \eqref{NonAutDAE} has a unique solution $(x(t),\lambda(t))$ for $|t-t_0|<\delta$ that satisfies the initial condition
\begin{equation}\label{InitialDataNonAutDAE}
x(t_0)=x_0,\,\,\,\,\,\dot x(t_0)=\dot x_0.
\end{equation}
Along this solution the full-rank condition
{\rm
\begin{equation}\label{FullRank}
\mbox{rank}\,\nabla_{\dot x}\,\varphi(t,x,\dot x)=m.
\end{equation}
}
holds for $|t-t_0|<\delta$.
\end{theorem}
Due to its importance for future purposes and for convenience of the reader we include the proof given in \cite{RigidBodyDAE} with slight modifications.
\begin{proof}
From \eqref{InitialCoeff} $(iii)$ it follows that rank $\nabla_{\dot x}\varphi(t,x,\dot x)=m$ in some open neighborhood of $(t_0,x_0,\dot x_0)$ in $\R\times\R^n\times\R^n$ (indeed, as we will see shortly, in the nonholonomic case $\varphi$ does not depend on $t$ and $\nabla_{\dot x}\varphi$ is full-rank everywhere in $\R^n\times\R^n$ by construction). This implies that dim ker $\nabla_{\dot x}\varphi=n-m$; whence the orthogonal projection $P\Arg$ of $\R^n$ onto ker $\nabla_{\dot x}\varphi\Arg$ is a smooth function in a neighborhood of $\Argo$.

Suppose that $x=x(t)\,,\,\lambda=\lambda(t)$ is a solution of \eqref{NonAutDAE} satisfying \eqref{InitialDataNonAutDAE}, and defined for all $t$ in some sufficiently small interval $\Delta=|t-t_0|$ around $t_0$. Then, in $\Delta$, we have
\begin{equation}\label{Proj}
P\Argt A(x(t))\ddot x(t)=P\Argt F\Argt.
\end{equation}
Moreover, the constraint $\varphi\Argt=0$ holds in $\Delta$; by differentiation we obtain
\begin{equation}\label{ConsDiff}
\nabla_{\dot x}\varphi\Argt\ddot x(t)=-\nabla_x\varphi\Argt\dot x(t)-\nabla_t\varphi\Argt.
\end{equation}
Note that $P\Argo$ is a linear isomorphism from ker $\nabla_{\dot x}\varphi\Arg$ onto ker $\nabla_{\dot x}\varphi\Argo$ for $\Arg$ close to $\Argo$. Thus, in $\Delta$,\eqref{Proj} is equivalent to
\begin{eqnarray}\label{proj0}
&P\Argo P\Argt A(x(t))\ddot x(t)=\\
&P\Argo P\Argt F\Argt\nonumber
\end{eqnarray}
Next, it is proved that the linear operator $\R^n\Flder \mbox{ker}\,\nabla_{\dot x}\varphi\Argo\times\R^m$ given by
\begin{equation}\label{Operator}
s\mapsto (P\Argo P\Arg A(x)\,s\,,\,\nabla_{\dot x}\varphi\Arg\,s),\,\,s\in\R^n,
\end{equation} 
is a linear isomorphism for $\Arg$ close to $\Argo$. By continuity it is enough to show the claim for $\Arg=\Argo$. As proof's strategy we calculate the kernel of the operator in order to show that it is null. Therefore, suppose $P\Argo A(x_0)\,s=0$ and $\nabla_{\dot x}\varphi\Argo\,s=0.$ The latter equation means that $s\in\,\mbox{ker}\,\nabla_{\dot x}\varphi\Argo$, and hence the former yields
\[
\bra P\Argo A(x_0)\,s\,,\,s\ket=\bra A(x_0)\,s\,,\,s\ket=0,
\]
where the symmetry of the projector is used. But, since $A$ is assumed to be possitive definite by \eqref{InitialCoeff} $(i)$, $s=0$ and the claim follows.

With this, we see that for $\Delta$ small enough the system of equations \eqref{proj0} and \eqref{ConsDiff} can be solved for $\ddot x(t)$ by
\begin{equation}\label{soODE}
\ddot x(t)=\Gamma\Argt,
\end{equation}
where $\Gamma$ is a smooth function. Hence, the assumed solution $x(t)$ of DAE \eqref{NonAutDAE} solves the second-order ODE \eqref{soODE} on $\R^n$ uniquely (the initial conditions \eqref{InitialDataNonAutDAE} fix a particular solution of \eqref{soODE}). Finally, since $\nabla_{\dot x}\varphi\Argt^{T}$ is injective, $\lambda(t)$ can be determined uniquely from \eqref{NonAutDAEa}.

Conversely, it has to be proved that the unique solution of \eqref{soODE} determined by \eqref{InitialDataNonAutDAE} indeed solves \eqref{NonAutDAE}. As a matter of fact, since the right hand side of \eqref{soODE} is the solution of the system \eqref{proj0}, \eqref{ConsDiff}, this solution must also satisfy the relations \eqref{ConsDiff} and \eqref{Proj}. In particular, by \eqref{Proj} we have
\begin{eqnarray*}
A(x(t))\ddot x(t)-F\Argt &\in& \mbox{ker}\,P\Argt\\
                         &=& \lp\mbox{ker}\,\nabla_{\dot x}\varphi\Argt\rp^{\perp}\\
                         &=&\mbox{rge}\, \nabla_{\dot x}\varphi\Argt^T.
\end{eqnarray*}
But every such element can be written in the form $\nabla_{\dot x}\varphi\Argt^T\lambda$ for a unique $\lambda=\lambda(t)\in\R^m$. This shows that 
\[
A(x(t))\ddot x(t)=F\Argt+\nabla_{\dot x}\varphi\Argt^T\lambda(t).
\]
Finally, \eqref{ConsDiff} ensures that $\varphi\Argt$ is constant, and hence, because $\varphi\Argo$ $=0$ due to \eqref{InitialCoeff} $(ii)$, we obtain $\varphi\Argt=0$. This proves that the solution of \eqref{soODE} indeed solves \eqref{NonAutDAE}.
\end{proof}
Without further assumptions, equations \eqref{NhDAE1} together with the initial conditions \eqref{InitialData} do not fit in the general setting established by theorem \ref{EUDAETh}. Therefore, in the following we restrict ourselves to Lagrangian functions, the matrix of second partial derivatives of which, i.e. 
\begin{equation}\label{metric}
\lp\frac{\der^2L}{\der\dot q^j\der\dot q^i}\rp,
\end{equation}
only depends on the configuration variables, and is regular ({\it regular} Lagrangians) as well as positive-definite ({\it normal} Lagrangians). In other words, \eqref{metric} represents a Riemannian metric on $Q$. Needless to say, these assumptions are completely consistent with most of the Lagrangian functions showing up in mechanics, for instance in the case of simple mechanical systems that we consider in the examples later on. Thus, comparing \eqref{NonAutDAE} and \eqref{NhDAE1} we set
\begin{equation*}
A_{ji}=\frac{\der^2L}{\der\dot q^j\der\dot q^i},\,\,\,\,\,\,\, F_i=-\frac{\der^2L}{\der q^j\der\dot q^i}\,\dot q^j+\frac{\der L}{\der q^i},\,\,\,\,\,\,\,\varphi^{\alpha}=\mu^{\alpha}_i(q)\,\dot q^i,
\end{equation*}
in order to ensure the local existence and uniqueness of the solutions $(q(t),$ $\lambda(t))$ of \eqref{NhDAE1}, since
\[
\mbox{rank}\,\nabla_{\dot q^i}\,\varphi^{\alpha}=\mbox{rank} (\mu^{\alpha}_i(q))=m,
\]
which holds by the definition of the nonholonomic problem, and as long as the initial condition \eqref{InitialData} respects the constraints \eqref{NhDAE1c}. As an important remark, note that theorem \ref{EUDAETh} ensures the solvability of more general, non-autonomous systems than the nonholonomic DAEs which we consider here. This fact will play a role in subsequent sections.

\subsection{Nonholonomic DAEs as ODEs on a manifold}
Our purpose in this subsection is to transform the DAE \eqref{NhDAE1} into an ODE on a manifold (see \cite{DAERheinboldt2,GeomTreat,Reich1,Reich2,DAERheinboldt1} for other approaches and similar constructions).
Generally speaking, if $\mathcal{M}$ is a submanifold of $\R^N$, i.e. $\mathcal{M}\subset\R^N$, and $y(t)$ is a differentiable curve contained in it, then (by definition of the tangent space) its time derivative satisfies $\dot y(t)\in T_{y(t)}\mathcal{M}$ for all $t$. On the other hand, a vector field on $\mathcal{M}$ is a $C^1-$mapping $f:\mathcal{M}\Flder\R^N$ such that $f(y)\in T_{y}\mathcal{M}$ for all $y\in\mathcal{M}$. For such a vector field 
\[
\dot y=f(y)
\]
is called a differential equation on the submanifold $\mathcal{M}$, and a function $y:I\Flder\mathcal{M}$ satisfying $\dot y(t)=f(y(t))$ for all $t\in I$ is called integral curve or simply solution of that equation. Now, our purpose is to construct a vector field on the  distribution $D$ that represents the nonholonomic DAE \eqref{NhDAE1}.

Setting $\dot q^i=v^i$, the nonholonomic DAE \eqref{NhDAE1} can be rewritten as
\begin{subequations}\label{NhDAE2}
\begin{align}
\dot q^i&=v^i,\label{NhDAE2a}\\
\frac{\der^2L}{\der v^j\der v^i}\,\dot v^j&=-\frac{\der^2L}{\der q^j\der v^i}\,v^j+\frac{\der L}{\der q^i}+\lambda_{\alpha}\,\mu^{\alpha}_{i}(q),\label{NhDAE2b}\\
\mu^{\alpha}_i(q)\,v^i&=0.\label{NhDAE2c}
\end{align}
\end{subequations}
As mentioned before, the matrix \eqref{metric} is considered to be regular and postive-definite. Denoting $\lp m_{ij}\rp=\lp\frac{\der^2L}{\der v^j\der v^i}\rp$, and its inverse by $\lp m^{ij}\rp$, i.e. $\lp m^{ik}m_{kj}\rp=\lp\delta^i_j\rp$, equations \eqref{NhDAE2} may be rewritten as the following DAE on $\R^{2n + m}$:
\begin{subequations}\label{NhDAE3}
\begin{align}
\dot y&=f(y)+\lambda_{\alpha}\lp\begin{array}{c} 0\\m^{-1}\nabla_{v}\varphi^{\alpha}(y)\end{array}\rp,\label{NhDAE3a}\\
\varphi^{\alpha}(y)&=0\label{NhDAE3b},
\end{align}
\end{subequations}
where $y=(q^i,v^i)^T\in\R^{2n}$, $\lambda=(\lambda_{\alpha})^T\in\R^{m}$, $\varphi^{\alpha}(y)=\mu^{\alpha}_i(q)\,v^i$, $\nabla_{v^i}=\frac{\der}{\der v^i}$ (analogously we will denote $\nabla_{q^i}=\frac{\der}{\der q^i}$) and
\begin{equation}\label{f}
f(y)=\lp f_q^i(y),f_v^i(y)\rp^T=\lp v^i\,\,,\,\,-m^{ik}\frac{\der^2L}{\der q^j\der v^k}v^j+m^{ij}\frac{\der L}{\der q^j}\rp^T.
\end{equation}
\begin{remark}
{\rm Note that equations \eqref{NhDAE3} define a very particular DAE strongly depending on the nonholonomic structure, a fact which is taken into account in the definition \eqref{f} of $f(y)$ and in $m^{-1}\nabla_{v}\varphi^{\alpha}(y)=\lp\frac{\der^2L}{\der v^i\der v^j}\rp^{-1}\mu^{\alpha}_j(q).$ }
\end{remark}
Next, we shall prove that \eqref{NhDAE3b} determines a submanifold in $\R^{2n}$. We could omit the following result just by invoking the nonholonomic structure which establishes by construction that $D$ is a smooth submanifold of  $TQ$. Nevertheless, we include it (see \cite{DAERheinboldt2,DAERheinboldt1} for the proof) for later use in remark \ref{deformation} where we allow a deformation of the nonholonomic constraints:
\begin{theorem}\label{Subma}
Let $\phi:U\subset\R^{l}\Flder\R^m$, $1\leq m<l$, be a $C^r-$mapping, $r\geq 1$, on an open set $U\subset\R^{l}$. Then the regularity set
{\rm
\begin{equation}\label{RegularitySet}
\mathcal{R}(\phi,U)=\lc x\in U\,,\,\mbox{rge}\,\nabla\phi(x)=\R^m\rc
\end{equation}
}
is open in $\R^l$, and for $0\in\phi\lp\mathcal{R}(\phi,U)\rp $, the regular solution set
\begin{equation}\label{Submanifold}
M=M(\phi,U)=\lc x\in\mathcal{R}(\phi,U)\,,\,\phi(x)=0\rc
\end{equation}
is a nonempty (sub-)manifold of $\R^l$ of class $C^r$ and dimension $l-m$.
\end{theorem}
Fixing $l=2n$ and $\phi=\varphi$, by that theorem it is straightforward to see that \eqref{NhDAE3b} determines a submanifold $M(\varphi,\R^{2n})\subset\R^{2n}$ since
\[
\mbox{rank}\,\nabla\varphi^{\alpha}=\mbox{rank}\lp\nabla_q\varphi^{\alpha}\,,\,\nabla_v\varphi^{\alpha}\rp=\mbox{rank}\lp \frac{\der\mu^{\alpha}_i}{\der q^j}v^i\,,\,\mu^{\alpha}_j\rp=m,
\]
for any $y\in\R^{2n}$ by the structure of the nonholonomic problem. Of course, in the nonholonomic context this submanifold $M(\varphi,\R^{2n})$ is nothing but the distribution $D$ introduced in $\S$\ref{NonHoloProblem}. When constructing an ODE on $D$ from the nonholonomic DAE, the following lemma will be of main importance. It follows from the fact that both $(m_{ij})$ and $(\mu_i^{\alpha})$ are full-rank matrices; nevertheless we give a geometrical proof which takes into account the structure of our nonholonomic system and which will be relevant in remark \ref{deformation}. 

\begin{lemma}\label{lemmaC}
The $m\times m$ matrix $\lp C^{\alpha\beta}\rp=\lp\mu^{\alpha}_im^{ij}\mu^{\beta}_j\rp$ is invertible.
\end{lemma}
\begin{proof}
Since $TQ$ is equipped with the Riemannian metric $m$ (recall that we are considering $\lp\frac{\der^2L}{\der v^i\der v^j}\rp$ to be regular and postive-definite), we can perform the decomposition $TQ=D\oplus D^{\perp}$. Here, $D^{\perp}$ is determined by the condition
\[
m(Z^{\alpha},Y)=0,\,\,\,\forall\,Y\in D. 
\]
Taking into account that $\bra\mu^{\alpha}\,,\,Y\ket=0$ for any $Y\in D$, where $\bra\cdot,\cdot\ket$ represents the canonical pairing between $TQ$ and $T^*Q$, the local expression of $Z^{\alpha}$ is given by
\begin{equation}\label{Zlocal}
Z^{\alpha}=m^{ij}\mu^{\alpha}_i\frac{\der}{\der q^j}.
\end{equation}
In order to prove that $\lp C^{\alpha\beta}\rp$ is invertible, we proceed to calculate its kernel, that is $\omega_{\beta}\in\R^{m}$ s.t. $C^{\alpha\beta}\omega_{\beta}=0$:
\[
C^{\alpha\beta}\omega_{\beta}=\mu^{\alpha}_im^{ij}\mu^{\beta}_j\omega_{\beta}=\bra\mu^{\alpha}_k dq^k\,,\,\omega_{\beta}m^{ij}\mu^{\beta}_j\frac{\der}{\der q^i}\ket=0.
\]
Taking into account \eqref{Zlocal}, we can write the previous expression as
\[
\omega_{\beta}\bra\mu^{\alpha}\,,\,Z^{\beta}\ket=0.
\]
Since $Z^{\beta}\in D^{\perp}$, $\bra\mu^{\alpha}\,,\,Z^{\beta}\ket$ is different from zero by definition; in consequence, the previous expression vanishes if and only if $\omega_{\beta}=0$. Therefore ker $\lp C^{\alpha\beta}\rp$ is trivial, and the claim is proved.
\end{proof}
We point out that the matrix $\lp C^{\alpha\beta}\rp$ only depends on $q$, a fact that follows from the $q-$dependence of the one-forms $\mu^{\alpha}$ and of the matrix $m$ (assumed from the beginning), and therefore of its inverse. Using this lemma, we can finally construct the vector field on $D$ defining an ODE as formulated in the following proposition.
\begin{proposition}\label{h}
The nonholonomic DAE \eqref{NhDAE3} induces an ODE on $D=M(\varphi,\R^{2n})$ given by
\begin{equation}\label{ODEManifold}
\dot x=h(x),
\end{equation}
where $x\in D\subset\R^{2n}$ and $h(x)\in T_xD\subset T\R^{2n}$ is defined by
\[
h(x):=f(x)+\lambda_{\alpha}(x)\lp\begin{array}{c} 0\\m^{-1}\nabla_{v}\varphi^{\alpha}(x)\end{array}\rp.
\]
$f(x)$ is defined in \eqref{f} and $\lambda_{\alpha}(x)$ is determined by
\begin{small}
\begin{equation}\label{lambdax}
\lambda_{\alpha}(x)=-C_{\alpha\beta}(x)\lp\frac{\der\mu^{\beta}_i(q)}{\der q^j}v^iv^j +\mu^{\beta}_i(q)\,m^{ij}\,\frac{\der L}{\der q^j}-\mu^{\beta}_i(q)\,m^{ik}\frac{\der^2 L}{\der q^j\der v^k}v^j\rp
\end{equation}
\end{small}
where $\lp C_{\alpha\beta}\rp$ is the inverse matrix of $\lp C^{\alpha\beta}\rp$ defined in lemma \ref{lemmaC} and $x=(q^i,v^i)^T$ s.t. $\mu^{\alpha}_i(q)v^i=0$.
\end{proposition}
Note that we view $y$ to be a point belonging to $\R^{2n}$ in \eqref{NhDAE3} while we write  $x$ in \eqref{ODEManifold} to stress the fact that the equation \eqref{ODEManifold} is defined on the submanifold $D\subset\R^{2n}$ and consequently, $x\in D$. We use $y$ in the next proof where we are determining the conditions providing \eqref{ODEManifold}.
\begin{proof}
Note that $\lp C_{\alpha\beta}\rp(x)$ is only $q-$de\-pen\-dent. The equation \eqref{NhDAE3b} determines a submanifold in $\R^{2n}$, the normal vector of which is defined by $\nabla\varphi^{\alpha}$. On the other hand, the right hand side of \eqref{NhDAE3a} defines a vector field on $T\R^{2n}$. With these two ingredients, we can establish the following perpendicularity condition:
\begin{equation}\label{Perp}
\bra\nabla\varphi^{\beta}(y)\,,\,f(y)+\lambda_{\alpha}\lp\begin{array}{c} 0\\m^{-1}\nabla_{v}\varphi^{\alpha}(y)\end{array}\rp\ket=0,
\end{equation}
where the one-form $\nabla\varphi^{\beta}$ is defined in coordinates by $\nabla\varphi^{\beta}=\lp\nabla_q\varphi^{\beta},\nabla_v\varphi^{\beta}\rp=\lp\frac{\der\varphi^{\beta}}{\der q^i},\frac{\der\varphi^{\beta}}{\der v^i}\rp$ for a fixed $\beta$. The previous equation determines $\lambda_{\alpha}$ in terms of $y\in\R^{2n}$, due to the invertibility result of lemma \ref{lemmaC} as we will see shortly, in such a way that the vector field 
\[
h(y)=f(y)+\lambda_{\alpha}(y)\lp\begin{array}{c} 0\\m^{-1}\nabla_{v}\varphi^{\alpha}(y)\end{array}\rp
\]
is tangent to $D\subset\R^{2n}$, that is, it belongs to $T_xD\subset T\R^{2n}$. The perpendicularity condition \eqref{Perp} can be written as
\begin{equation}\label{Perp2}
\bra\nabla_q\varphi^{\beta}(y),f_q(y)\ket+\bra\nabla_v\varphi^{\beta}(y), f_v(y)\ket+\lambda_{\alpha}\bra\nabla_v\varphi^{\beta}(y)\,,m^{-1}\,\nabla_v\varphi^{\alpha}(y)\ket=0,
\end{equation}
where $m^{-1}\nabla_v\varphi^{\alpha}$ stands for the vector $m^{ij}\frac{\der\varphi^{\alpha}}{\der v^j}$.  We have also used the shorthand notation $f_q^i(y)=v^i$, and $f_v^i(y)=-m^{ik}\frac{\der^2L}{\der q^j\der v^k}v^j+m^{ij}\frac{\der L}{\der q^j}$ according to \eqref{f}. Moreover, the elementwise expression of the matrix 
\[
\bra\nabla_v\varphi^{\beta}(y),m^{-1}\,\nabla_v\varphi^{\alpha}(y)\ket
\] 
is $\lp\mu^{\beta}_im^{ij}\mu^{\alpha}_j\rp$ which, according to lemma \ref{lemmaC}, is invertible. Therefore, after a straightforward computation, from \eqref{Perp2} we arrive at \eqref{lambdax}. This finishes the proof.
\end{proof}
Note again that this development depends strongly on the nonholonomic structure. Namely, the construction of \eqref{ODEManifold} depends on the invertibility of $\lp C^{\alpha\beta}\rp$, which at the same time depends on the full-rank condition of $\lp\mu^{\alpha}_i(q)\rp$ prescribed by our nonholonomic setting.

Let us finaly introduce local coordinates $\xi\in\R^{2n-m}$ and the mapping $\psi:\R^{2n-m}\Flder D\subset\R^{2n}$ such that for $x\in D\subset\R^{2n}$, $x=\psi(\xi)$. From equation \eqref{ODEManifold}, we can derive the following ODE on $\R^{2n-m}$
\begin{equation}\label{ODExi}
\dot\xi=\lp\nabla_{\xi}\psi\rp^{-1}(\xi)\,h\lp\psi(\xi)\rp,
\end{equation}
with initial condition $\xi_0=\psi\big|_{D}^{-1}(x_0)$, where $x_0=x(t_0)$. In general the pseudo inverse matrix $\lp\nabla_{\xi}\psi\rp^{-1}$ is not unique. We shall give an explicit matrix represenation for $\lp\nabla_{\xi}\psi\rp^{-1}$  below, where we introduce an adapted set of coordinates for the nonholonomic distribution $D$ prescribed by the Ehresmann connection.

\section{Interpolation within the constraint submanifold}\label{Interpo}

Our aim in this section is to investigate the possibility of interpolating a curve connecting two points $x_1,x_2\in D$ which remains within $D$. As mentioned in the introduction, this fact will play an important role in order to embed a numerical method evolving on $D$ into the evolution of a non-autonomous perturbation of the nonholonomic ODE. For this purpose, we first establish a useful local form of the nonholonomic constraints by means of an Ehresmann connection.

\subsection{Local description of the constraint submanifold in terms of the Ehresmann connection}\label{Conecction}

We briefly review some basics on Ehresmann connections (for more details we refer to \cite{Bl,Koon}). Assume that there is a bundle structure with projection $\pi:Q\Flder R$ for the manifold $Q$; the  manifold $R$ is called the base. We call the kernel of $T_q\pi$ at any point $q\in Q$ the {\it vertical space} denoted by $\mathcal{V}_q$. An {\it Ehresmann connection} $A$ is a vertical vector-valued one-form on $Q$, which satisfies
\begin{enumerate}
\item $A_q:T_qQ\Flder\mathcal{V}_q$ is a linear map at each point $q\in Q$,
\item $A$ is a projection, i.e.,  $A(v_q)=v_q$, for all $v_q\in\mathcal{V}_q$.
\end{enumerate}
Thus, we can split the tangent space at $q$ such that $T_qQ=\mathcal{H}_q\oplus\mathcal{V}_q$, where $\mathcal{H}_q=\mbox{ker}\,A_q$ is the horizontal  space at $q$.

Consider the nonholonomic distribution $D$, which, as introduced in $\S$\ref{NonHoloProblem}, locally is given by 
\[
D(q)=\lc v_q\in TQ\,|\,\bra \mu^{\alpha},v_q\ket=0,\,\alpha=1,...,m\rc,
\]
where $\mu^{\alpha}$ are $m$ linearly independent one-forms that form a basis for the annihilator $D^{\circ}\subset T^*Q$. Let us choose an Ehresmann connection $A$ on $Q$ in such a way that $\mathcal{H}_q=D(q)$. In other words, we assume that the connection is chosen such that the constraints are given by $A\cdot v_q=0$.

Using the bundle coordinates $q=(y^a,y^{\alpha})\in\R^{n-m}\times\R^m$, $a=1,...,n-m$, $\alpha=n-m+1,...,n$, the coordinate expression of $\pi$ is just a projection onto the factor $y^a$, and the connection $A$ can be locally expressed by a vector-valued differential one-form $\mu^{\alpha}$ as
\[
A=\mu^{\alpha}\otimes\frac{\der}{\der y^{\alpha}},\,\,\,\mu^{\alpha}(q)=dy^{\alpha}+A_a^{\alpha}(y^a,y^{\alpha})\,dy^a. 
\]
Let
\[
v_q=v^{a}\frac{\der}{\der y^{a}}+v^{\alpha}\frac{\der}{\der y^{\alpha}}
\]
be an element of $T_qQ$. Then
\[
A(v_q)=(v^{\alpha}+A_a^{\alpha}\,v^a)\otimes\frac{\der}{\der y^{\alpha}}.
\]
Given a vector $v_q\in T_qQ$,  we denote its vertical part by
\[
\mbox{ver}\,v_q=A\cdot v_q,
\]
in local corrdiantes given by $(v^{a},v^{\alpha})\mapsto(0,v^{\alpha}+A_a^{\alpha}v^a)$, and its horizontal part by
\[
D\ni\mbox{hor}\,v_q=v_q-A\cdot v_q,
\]
with local expression $(v^{a},v^{\alpha})\mapsto(v^a,-A_a^{\alpha}v^a)$.

Therefore, the Ehresmann connection allows us to choose an adapted set of coordinates for the constraint distribution $D$, namely $\xi^I=(q^i,v^a)^T$, $I=1,...,2n-m$, where $\psi:\R^{2n-m}\Flder D\subset\R^{2n}$, 
\begin{equation}\label{Embedd}
\psi(q^i,v^a)=\lp q^i,v^a,-A^{\alpha}_a(q)\,v^a\rp,
\end{equation}
denotes the coordinate map. This is what we mean by $D$-adapted coordinates. Moreover, we introduce the $(2n\times\,2n-m)$ matrix $\nabla_{\xi}\psi=\lp\frac{\der\psi^i_q}{\der\xi^I},\frac{\der\psi^i_v}{\der\xi^I}\rp$, with $\psi^i_q=q^i$, $\psi^i_v=(v^a,-A^{\alpha}_a(q)\,v^a)^T$. More concretely, we have
\begin{equation}\label{EhresJota}
\nabla_{\xi}\psi=\lp\begin{array}{cc}
\delta^i_j&0^i_b\\\\
0^a_j&\delta^a_b\\\\
-\frac{\der A^{\alpha}_a(q)}{\der q^j}v^a &-A^{\alpha}_b(q)
 \end{array}
\rp.
\end{equation}
As mentioned before, $(\nabla_{\xi}\psi)$ admits a set of pseudo-inverses (left-inverses). Hereafter, we choose the following one, for reasons that will become clear later:
\begin{equation}\label{Inverse}
\lp\nabla_{\xi}\psi\rp^{-1}=\lp
\begin{array}{ccc}
\delta^i_j & 0_b^i&0^i_{\alpha}\\\\
0^a_j&\delta^a_b &0^a_{\alpha}
\end{array}
\rp
\end{equation}

\subsection{Interpolation}\label{Interpol}

Consider the $C^{\infty}$ cut-off function
\[
\chi_0:\R\Flder[0,1]
\]
such that
\begin{equation}\label{CutOff}
\begin{array}{c}
\chi_0(\tau)\equiv1\,\,\mbox{for}\,\,\tau\leq0,\\
\chi_0(\tau)\equiv0\,\,\mbox{for}\,\,\tau\geq1,
\end{array}
\end{equation}
and assume that $\chi_0$ is real analytic for $\tau\neq0,1$.  Also, denote $\chi_1(\tau):=1-\chi_0(\tau)$. For instance, we could take
\[
\chi_0(\tau)=\lp1+\mbox{tanh}\lp\mbox{cot}(\pi\tau)\rp\rp/2,\,\,\,0<\tau<1.
\]
Consider two points $x_1,x_2\in D$, which, with respect to the coordinates induced by the Ehresmann connection, are locally represented by
\begin{equation}\label{yD}
\begin{array}{c}
x_1=(q^i_1,v^a_1,v^{\alpha}_1=-A^{\alpha}_a(q)\,v^a_1)^T\in D\subset\R^{2n},\\\\
x_2=(q^i_2,v^a_2,v^{\alpha}_2=-A^{\alpha}_a(q)\,v^a_2)^T\in D\subset\R^{2n}.
\end{array}
\end{equation}
Consider also the real interval $[0,\epsilon]\subset\R_+$. 
\begin{definition}
Define $C^{\infty}$ curves $q^i:[0,\epsilon]\Flder\R$, $v^a:[0,\epsilon]\Flder\R$ by
\begin{equation}
\begin{array}{c}
q^i(t)=\chi_0(t/\epsilon)\,q^i_1+\chi_1(t/\epsilon)\,q^i_2,\\\\
v^a(t)=\chi_0(t/\epsilon)\,v^a_1+\chi_1(t/\epsilon)\,v^a_2.
\end{array}
\end{equation} 
and the $C^{\infty}$ curve $c:[0,\epsilon]\Flder\R^{2n}$ by
\begin{equation}
c(t)=(q^i(t),v^a(t),v^{\alpha}(t)=-A^{\alpha}_a(q(t))\,v^a(t))^T
\end{equation} 
for $i=1,...,n$, $a=1,...,n-m$ and $\alpha=1,..,,m$.
\end{definition}
\begin{proposition}\label{InterpolationD}
$c(t)\subset D$ for $t\in[0,\epsilon]$.
\end{proposition}
\begin{proof}
The proof follows directly from the decomposition induced by the Ehresmann connection.
\end{proof}
Note that $c(0)=x_1$ and $c(\epsilon)=x_2$ according to \eqref{yD}. Thus, for any two points belonging to $D$, we have constructed an interpolating curve $c(t)\subset D$ for $t\in[0,\epsilon]$ connecting them. (A different and interesting procedure to interpolate any number of points on a manifold has been proposed in \cite{Leite}.)

\section{Discretization of the Lagrangian nonholonomic problem}\label{DLNP}
In this section we describe in detail the result from \cite{FiSch} already mentioned in the introduction, which can be phrased as follows: {\it Any $p-$th order discretization of an ODE can equivalently be viewed as the time$-\epsilon$ period map of a suitable $\epsilon-$periodic non-autonomous perturbation of the original ODE.} Moreover, we apply this result to the nonholonomic case, in particular to the nonholonomic ODE \eqref{ODExi}.
\subsection{Discretization of ODEs}

Consider the following system of ordinary differential equations
\begin{equation}\label{ODE}
\dot z(t)=\tf(z(t)),\,\,\,\,\,z(t)\in\R^N,
\end{equation}
where $\tf$ is a real analytic vector field, and a $p-$th order discretization of step-size $\epsilon$ given by
\begin{equation}\label{ODEdisc}
z_{k+1}=\Phi(\epsilon,z_k),\,\,\,|\epsilon|\leq\epsilon^0,\,\,z_k\in\R^N.
\end{equation}
$\Phi$ is supposed to be a diffeomorphism with respect to $z$, real analytic with respect to $\epsilon$ and $z$. Denote the (local) flow of \eqref{ODE} by
\begin{equation}\label{ODEflow}
z(t)=F(t,z(0)).
\end{equation}
If $\Phi$ is a discretization of order $p$, then there exists a continuous and increasing function $C:[0,\infty)\Flder[0,\infty)$ such that
\begin{equation}\label{porder}
|\Phi(\epsilon,z)-F(\epsilon,z)|\leq C(|z|)\epsilon^{p+1}
\end{equation}
holds, for some $p\geq 1$ and for all $\epsilon$, $z$ for which the left-hand side is defined.

In the following proposition we prove the {\it embeddability} property of such a numerical scheme into the evolution map 
\[
z(t)=G(t,s,\epsilon,z(s))
\]
of a non-autonomous ODE of the form
\begin{equation}\label{pertODE}
\dot z(t)=\tf(z(t))+\epsilon^p g(\epsilon,t/\epsilon,z(t)).
\end{equation}

\begin{proposition}\label{propoDisc}
There exists a vector field
\[
g=g(\epsilon,\tau,z),
\]
$\epsilon_0>0$, and a continous non-increasing function 
\[
\rho:[0,\epsilon_0]\Flder[0,\infty]
\]
with $\rho(0)=\infty$, such that the following statements hold
\begin{itemize}
\item[$i)$] $g(\epsilon,\tau,z)\in\R^N$ is defined for all real $\tau$, $0\leq\epsilon\leq\epsilon_0$, and all real $z$ with $|z|<\rho(\epsilon)$;

\item[$ii)$] $g$ is $C^{\infty}-$smooth in all variables, and $g$ and all its $\tau-$derivatives are analytic in $(\epsilon,z)$;

\item[$iii)$] $g$ has period $1$ in $\tau$;

\item[$iv)$] $G(\epsilon,0,\epsilon,z)=\Phi(\epsilon,z)$.
\end{itemize}
Statements (ii)-(iv) hold for all $\epsilon$, $\tau$, $z$ satisfying (i).
\end{proposition}
We present part of the proof of the interpolation procedure. In particular we only consider the case $\tau\geq\ 0$. For more details see \cite{FiSch}.
\begin{proof}
We focus on the construction of $G$ and the time periodicity of $g$. For sake of simplicity we omit the arguments $(\epsilon,z)$. To define an evolution map, we put $G(t,s):=G(t,0)\circ G(s,0)^{-1}$, and extend to $t\geq\epsilon$ successively.

The idea is to interpolate between the identity map and $\Phi$ by a curve $G(t,0)$ ($0\leq t\leq \epsilon$), in the space of diffeomorphisms. Employing the $C^{\infty}$ cut-off functions $\chi_0,\,\chi_1$ introduced in $\S$\ref{Interpol}, namely for $0\leq t\leq \epsilon$ and $|y|<\rho(\epsilon)$, we set
\begin{equation}\label{Gdef}
G(t,\epsilon,y):=\chi_0(t/\epsilon)\,F(t,y)+\chi_1(t/\epsilon)\,F(t-\epsilon,\Phi(\epsilon,y)),
\end{equation}
where $F$ is the local flow \eqref{ODEflow}, $\Phi$ is the discrete flow \eqref{ODEdisc} and $y$ stands for the initial value of $z$, i.e. $y=z(0)$. With this definition, it holds that
\begin{eqnarray*}
G(0,\epsilon,y)&=&F(0,y)=y,\\
G(\epsilon,\epsilon,y)&=&F(0,\Phi(\epsilon,y))=\Phi(\epsilon,y).
\end{eqnarray*}
Let $[\tau]$ denote the largest integer not exceeding $\tau\in\R$, and let $\Phi^k(\epsilon,\cdot)$ denote the $k-$th iterate of the discrete flow map, $k\geq1$. Then we extend our definition of $G$ to all $t\geq0$, by putting
\begin{equation}\label{period0}
G(t,\epsilon,y):=G(t-[t/\epsilon]\,\epsilon,\epsilon,\Phi^{[t/\epsilon]}(\epsilon,y)).
\end{equation}
This definition implies $G(t,\epsilon,G(k\epsilon,\epsilon,y))=G(t+k\epsilon,\epsilon,y)$ for all $t\geq0,\,k\in\N,\,\epsilon,\,y$ as is appropriate for the evolution map of an $\epsilon-$periodic, non-autonomous ODE system. Obviously the curve $t\mapsto G(t,\epsilon,y),\,\,t\geq0$, is $C^{\infty}$, and thus represents a  $C^{\infty}-$interpolation of the discrete forward orbit
\[
\Phi^k(\epsilon,y)=G(k\epsilon,\epsilon,y),\,\,\,k\in\N.
\]
Similarly, one can extend the definition of G to all $t\leq0$.  
In order to define the perturbation $g$, we switch to the scaled time variable $\tau=t/\epsilon$. Let
\begin{equation}\label{Gtilde}
\tilde G(\tau,\epsilon,y):=G(\epsilon\tau,\epsilon,y)=\chi_0(\tau)\,F(\epsilon\tau,y)+\chi_1(\tau)\,F(\epsilon\tau-\epsilon,\Phi(\epsilon,y)),
\end{equation}
the second equality holding only for $0\leq\tau\leq 1$. Then the perturbation $g$ is defined by
\begin{equation}\label{gdefinition}
g(\epsilon,\tau,z):=\epsilon^{-p}\lp -f(z)+\frac{1}{\epsilon}\nabla_{\tau}\tilde G(\tau,\epsilon,y)\rp,
\end{equation}
where $y=y(\epsilon,\tau,z)$ is given implicitly by 
\begin{equation}\label{ztilde}
z=\tilde G(\tau,\epsilon,y).
\end{equation}
In \cite{FiSch} it is proved that this last equality can be solved for $y$ providing $y=y(\epsilon,\tau,z)$, which is defined in a possibly reduced domain of the form $0\leq\tau\leq 1$, $0\leq\epsilon\leq\epsilon_0$, $|y|<\rho(\epsilon)$. To prove periodicity, note that
\begin{equation}\label{period1}
\Phi(\epsilon,y(\tau+1,\epsilon,x))=y(\tau,\epsilon,x)
\end{equation}
because \eqref{period0} implies
\[
\tilde G(\tau,\epsilon,\Phi(\epsilon,y(\tau+1,\epsilon,z)))=\tilde G(\tau+1,\epsilon,y(\tau+1,\epsilon,z))=z.
\]
Therefore we concluce from \eqref{period0},\eqref{gdefinition},\eqref{ztilde},\eqref{period1}
\begin{small}
\begin{eqnarray*}
\epsilon^{p+1}\lp g(\epsilon,\tau+1,z)-g(\epsilon,\tau,z)\rp &=&\nabla_{\tau}\tilde G(\tau+1,\epsilon,y(\tau+1,\epsilon,z))-\nabla_{\tau}\tilde G(\tau,\epsilon,y(\tau,\epsilon,z))\\
&=&\nabla_{\tau}\tilde G(\tau,\epsilon,\Phi(\epsilon,y(\tau+1,\epsilon,z)))-\nabla_{\tau}\tilde G(\tau,\epsilon,y(\tau,\epsilon,z))\\
&=&\nabla_{\tau}\tilde G(\tau,\epsilon,y(\tau,\epsilon,z))-\nabla_{\tau}\tilde G(\tau,\epsilon,y(\tau,\epsilon,z))=0.
\end{eqnarray*}
\end{small}
This proves the $1-$periodicity for $\epsilon>0$.
\end{proof}
Next, we apply proposition \ref{propoDisc} to the nonholonomic ODE \eqref{ODEManifold}. Obviously, it is sufficient to do this in coordinates.
\begin{corollary}\label{Coro}
Consider the ODE \eqref{ODExi}, i.e.
\[
\dot\xi=\lp\nabla_{\xi}\psi\rp^{-1}(\xi)\,h\lp\psi(\xi)\rp.
\]
According to proposition \ref{propoDisc}, any $p-$th order discretization of equations \eqref{ODExi} can be viewed as the time$-\epsilon$ map of a suitable $\epsilon-$periodic non-autonomous perturbation, namely
\begin{equation}\label{DiscretizationD}
\dot\xi(t)=\lp\nabla_{\xi}\psi\rp^{-1}(\xi(t))\,h\lp\psi(\xi(t))\rp+\epsilon^pg(\epsilon,t/\epsilon,\xi(t)).
\end{equation} 
\end{corollary}

Note that in the coordinates $(q^i,v^a)$  induced by the Ehresmann connection, the nonholonomic ODE \eqref{ODExi} reads as follows: 
\begin{subequations}\label{ODEred}
\begin{align}
\dot q^a&=v^a,\label{ODEreda}\\
\dot q^{\alpha}&=-A^{\alpha}_a(q)\,v^a,\label{ODEredb}\\
\dot v^a&=\tilde f_v^a(q^i,v^a)+\tilde\lambda_{\alpha}(q^i,v^a)\,m^{aj}\mu^{\alpha}_j(q),\label{ODEredc}
\end{align}
\end{subequations}
Here $f_q(x)$ is defined in \eqref{f} and $\lambda_{\alpha}(x)$ in \eqref{lambdax}, while $\tilde f_v^a(q^i,v^a)=f_v^a(\psi(q^i,v^a))$ and $\tilde\lambda_{\alpha}(q^i,v^a)=\lambda_{\alpha}(\psi(q^i,v^a))$.

We would like to stress the fact, that any $p-$th order discretization of the nonholonomic ODE corresponds to a rather special discretization of the original nonholonomic problem, for instance, respresented by the DAE in \eqref{NhDAE3}. We clarify the relevant notion of $p-$th order discretization of the nonholonomic problem by means of the following definition.

\begin{definition}\label{pdisc}
By a $(p,s)$ order discretization of the nonholonomic problem we understand a sequence of points $\lc(q_k,v_{q_k},\lambda_k)\rc\in\R^{2n+m}$, $k=0,1,...,N-1,N$, s.t. $v_{q_k}\in D_{q_k}$ (in other words $\mu^{\alpha}(q_k)\,v_{q_k}=0$ for any $k$), and, moreover, $|q(t_k+\epsilon)-q_{k+1}|\sim O(\epsilon^{r+1}),\,\,|v(t_k+\epsilon)-v_{q_{k+1}}|\sim O(\epsilon^{l+1})$ with {\rm min}$(r,l)=p$ and, moreover, $|\lambda(t_k+\epsilon)-\lambda_{k+1}|\sim O(\epsilon^{s+1})$ with $s\geq0$. The dynamics of $(q(t),v(t))$ is determined by \eqref{ODEManifold}, while $\lambda(t)$ is determined by \eqref{lambdax}.
\end{definition}
\begin{remark}\label{rmkpdisc}
{\rm
As mentioned above, the Lagrange multipliers $\lambda$ determine the reaction forces due to the constraints. This is the reason why we include $\lambda$ in the definition \ref{pdisc}. Obviously, if we discretize \eqref{ODEManifold} and define the discretized values $\lambda_k$ of $\lambda$ by inserting $(q_k,v_{q_k})$ for $q$ and $v$ in formula \eqref{lambdax}, then, by Lipschitz continuity of the right hand side, the accuracy of the discretization with respect to $\lambda$ is of the same order $p$ as with respect to $(q,v)$.
}
\end{remark}
Due to the definition of the vector field $h$ in \eqref{ODEManifold}, it is clear that discretizations with the properties just stated in definition \ref{pdisc} generate  $p-th$ order discretizations of the ODEs in \eqref{ODEManifold}, \eqref{ODExi}, and \eqref{ODEred}, respectively, and vice versa. We will devote the final sections to find a systematic way of constructing this kind of integrators and, moreover, to carefully treat some examples.

\subsection{Perturbation of the nonholonomic DAE}
In this section, we relate the previous developments to the nonholonomic DAE in \eqref{NhDAE3}. To this end, we multiply equation  \eqref{DiscretizationD} by $\nabla_{\xi}\psi$ from the left to obtain 
\begin{equation}\label{perturbedNhODExi}
\nabla_{\xi}\psi(\xi(t))\,\dot\xi(t)=h\lp\psi(\xi(t))\rp+\epsilon^p\nabla_{\xi}\psi(\xi(t))\,g(\epsilon,t/\epsilon,\xi(t)),
\end{equation}
which, taking into account that $x=\psi(\xi)$, leads to the following ODE on $D$:
\begin{equation}\label{perturbedNhODE}
\dot x=f(x)+\lambda_{\alpha}(x)\lp\begin{array}{c} 0\\m^{-1}(x)\nabla_{v}\varphi^{\alpha}(x)\end{array}\rp+\epsilon^p\,\nabla_{\xi}\psi(\xi_x)\,g(\epsilon,t/\epsilon,\xi_x),
\end{equation}
where $\xi_x:=\psi\big|_D^{-1}(x)$ s.t. $x=\psi(\xi_x)\in D\subset\R^{2n}$. Equation \eqref{perturbedNhODE} follows directly from \eqref{perturbedNhODExi} and the expression of $h(x)$ given in proposition \ref{h}. For the sake of clarity we set 
\begin{subequations}\label{tildeg}
\begin{align}
\tilde g_q^i(\epsilon,t/\epsilon,x)=\frac{\der\psi_q^i}{\der\xi^I}(\xi_x)\,g^I(\epsilon,t/\epsilon,\xi_x),\label{tildega}\\
\tilde g_v^i(\epsilon,t/\epsilon,x)=\frac{\der\psi_v^i}{\der\xi^I}(\xi_x)\,g^I(\epsilon,t/\epsilon,\xi_x),\label{tildegb}
\end{align}
\end{subequations}
where we are using the notation introduced in \eqref{Embedd}. Thus, the equation \eqref{perturbedNhODE} becomes
\begin{equation}\label{Odetilde}
\dot x=f(x)+\lambda_{\alpha}(x)\lp\begin{array}{c} 0\\m^{-1}(x)\nabla_{v}\varphi^{\alpha}(x)\end{array}\rp+\epsilon^p\,\lp\begin{array}{c}\tilde g_q(\epsilon,t/\epsilon,x)\\\tilde g_v(\epsilon,t/\epsilon,x)\end{array}\rp,
\end{equation}
where, as before, $x=(q,v)^T$ such that $\mu^{\alpha}_i(q)\,v^i=0$.

The following lemma will be useful below. 
\begin{lemma}\label{lemmaMinv}
$(m_{ij})=\lp\mu^{\alpha}_i\,C_{\alpha\beta}\,\mu^{\beta}_j\rp$, where $(C_{\alpha\beta})$ is the inverse matrix of $(C^{\alpha\beta})$ as provided by lemma \ref{lemmaC}.
\end{lemma}
\begin{proof}
We give the proof by a direct calculation in coordinates. As proved in lemma \ref{lemmaC}, the matrix $\lp C^{\alpha\beta}\rp$ is invertible; therefore $\lp C_{\alpha\gamma}\,\mu^{\gamma}_i\,m^{ij}\,\mu^{\beta}_j\rp=(\delta_{\alpha}^{\beta})$. Multiplying this relation by the matrix $(\mu^{\alpha}_k)$ from the left, we obtain
\[
\lp\mu^{\alpha}_k\, C_{\alpha\gamma}\,\mu^{\gamma}_i\,m^{ij}\,\mu^{\beta}_j\rp=(\mu^{\beta}_k).
\]
Rearranging the terms we get
\[
\lp m^{ji}\,\mu^{\gamma}_i\, C_{\gamma\alpha}\,\mu^{\alpha}_k\,\mu^{\beta}_j\rp=(\mu^{\beta}_k),
\]
where the symmetry of $m^{-1}$ and $C^{-1}$ has been used. From the last expression it follows that $\lp m^{ji}\,\mu^{\gamma}_i\, C_{\gamma\alpha}\,\mu^{\alpha}_k\rp=(\delta^j_k)$. Thus
\[
\lp\mu^{\gamma}_l\,C_{\gamma\alpha}\,\mu^{\alpha}_j\rp=(m_{lj})
\]
and the claim follows.
\end{proof}

\begin{remark}\label{LargeRemark}
{\rm
Obviously, by the derivation above, the ODE in \eqref{Odetilde} is equivalent to an $\epsilon-$periodic non-autonomously perturbed nonholonomic DAE in $R^{2n + m}$ of the following form
\begin{subequations}\label{PerturbedDAE}
\begin{align}
\dot y&=f(y)+\lambda_{\alpha}\lp\begin{array}{c} 0\\m^{-1}(y)\nabla_{v}\varphi^{\alpha}(y)\end{array}\rp+\epsilon^p\,\hat g(\epsilon,t/\epsilon,y),\label{PerturbedDAEa}\\
\varphi^{\alpha}(y)&=0\label{PerturbedDAEb},
\end{align}
\end{subequations}
where $y=(q^i,v^i)^T\in\R^{2n}$, $\varphi^{\alpha}(y)=\mu^{\alpha}_i(q)\,v^i$ and $f(y)$ is given by \eqref{f} (note that we use the $y-$notation here, since we are considering again an arbitrary point $y$ in $\R^{2n}$; we will switch back to the $x-$notation whenever we are dealing with equations on $D$). 

Moreover, by appropriately modifying the expression for the Lagrange multipliers $\lambda_{\alpha}$ in proposition \ref{h}, for any perturbation $\epsilon^p\,\hat g(\epsilon,t/\epsilon,y)$ in \eqref{PerturbedDAE}, that DAE induces an ODE as in \eqref{Odetilde}. In particular, this 
implies the unique local solvability of the standard initial value problem corresponding to \eqref{PerturbedDAE} (cf. the proof of theorem \ref{EUDAETh}). To show this, we first solve the following equation for $\lambda_{\alpha}$

\begin{equation}\label{hatperp}
\bra\nabla\varphi^{\beta}(y)\,,\,f(y)+\lambda_{\alpha}\lp\begin{array}{c} 0\\m^{-1}(y)\nabla_{v}\varphi^{\alpha}(y)\end{array}\rp+\epsilon^p\,\hat g(\epsilon,t/\epsilon,y)\ket=0,
\end{equation}
which is nothing but the perpendicularity condition ensuring that the vector field $\hat f(y)$ (where $\hat f(y)$ denotes the right hand side of \eqref{PerturbedDAEa}) is tangent to $D$ and therefore belongs to $TD$. Omitting, for simplicity, the arguments of the involved functions, we obtain

\begin{eqnarray}
\lambda_{\alpha}&=&-C_{\alpha\beta}\lp\bra\nabla_q\varphi^{\beta},f_q\ket+\bra\nabla_v\varphi^{\beta},f_v\ket\rp\label{hatLambday}\\
&-&\epsilon^p\,C_{\alpha\beta}\lp\bra\nabla_q\varphi^{\beta},\hat g_q\ket+\bra\nabla_v\varphi^{\beta},\hat g_v\ket\rp,\nonumber
\end{eqnarray}
where $\lp C_{\alpha\beta}\rp$ is again the inverse matrix of $\lp C^{\alpha\beta}\rp$ introduced in lemma \ref{lemmaC}. Moreover, we recall that $-C_{\alpha\beta}(x)\lp\bra\nabla_q\varphi^{\beta}(x),f_q(x)\ket+\bra\nabla_v\varphi^{\beta}(x),f_v(x)\ket\rp=\lambda_{\alpha}(x)$ as defined in the proof of proposition \ref{h}. Thus, the equation \eqref{hatLambday} reduces to
\[
\lambda_{\alpha}=\lambda_{\alpha}(x)-\epsilon^p\,C_{\alpha\beta}\lp\bra\nabla_q\varphi^{\beta},\hat g_q\ket+\bra\nabla_v\varphi^{\beta},\hat g_v\ket\rp.
\]
Introducing this last expression into \eqref{PerturbedDAEa} we arrive at
\begin{eqnarray}
\dot x&=&f(x)+\lambda_{\alpha}(x)\lp\begin{array}{c} 0\\m^{-1}(x)\nabla_{v}\varphi^{\alpha}(x)\end{array}\rp\label{ODEPert}\\
&+&\epsilon^p\,\lp\begin{array}{c} \hat g_q(\epsilon,t/\epsilon,x)\\-C_{\alpha\beta}\lp\bra\nabla_q\varphi^{\beta},\hat g_q\ket+\bra\nabla_v\varphi^{\beta},\hat g_v\ket\rp\,m^{-1}(x)\,\nabla_{v}\varphi^{\alpha}(x)+\hat g_v(\epsilon,t/\epsilon,x)\end{array}\rp.\nonumber
\end{eqnarray}
The term
\[
-C_{\alpha\beta}\lp\bra\nabla_q\varphi^{\beta},\hat g_q\ket+\bra\nabla_v\varphi^{\beta},\hat g_v\ket\rp\,m^{-1}(x)\,\nabla_{v}\varphi^{\alpha}(x)+\hat g_v(\epsilon,t/\epsilon,x)
\]
can be simplified taking into account the result in lemma \ref{lemmaMinv} as we show next. The term
\[
C_{\alpha\beta}\bra\nabla_v\varphi^{\beta},\hat g_v\ket\,m^{-1}\nabla_v\varphi^{\alpha}
\] 
can be written in coordinates as
\[
C_{\alpha\beta}\,\mu^{\beta}_k\,\hat g_v^k\,m^{ij}\,\mu^{\alpha}_j=m^{ij}\,\mu^{\alpha}_j\,C_{\alpha\beta}\,\mu^{\beta}_k\,\hat g_v^k.
\]
In lemma \ref{lemmaMinv} it has been proved that  $\lp \mu^{\alpha}_j\,C_{\alpha\beta}\,\mu^{\beta}_k\rp=(m_{jk})$; thus, the last equation yields
\[
m^{ij}\,m_{jk}\,\hat g_v^k=\delta^i_k\,\hat g_v^k=\hat g^i_v.
\]
Introducing this into \eqref{ODEPert}, we arrive at
\begin{eqnarray}
\dot x=f(x)&+&\lambda_{\alpha}(x)\lp\begin{array}{c} 0\\m^{-1}(x)\nabla_{v}\varphi^{\alpha}(x)\end{array}\rp\label{ODEPert2}\\
&+&\epsilon^p\,\lp\begin{array}{c} \hat g_q(\epsilon,t/\epsilon,x)\\-C_{\alpha\beta}\bra\nabla_q\varphi^{\beta},\hat g_q(\epsilon,t/\epsilon,x)\ket\,m^{-1}(x)\,\nabla_{v}\varphi^{\alpha}(x)\end{array}\rp.\nonumber
\end{eqnarray}
}
\end{remark}
Of course, it would be ideal for a further analysis, if the perturbed nonholonomic DAE in \eqref{PerturbedDAE} admitted a Lagrangian structure of some kind. The question, when this is true and how this is related to specific properties of both, the underlying discretization of the unperturbed problem as well as the interpolation procedure presented in the previous section, is left as a subject of further study. 
Replacing $D$ by some kind of perturbed distribution will be another issue in future work.

\section{Nonholonomic integrators based on the discretizaton of the Lagrange-d'Alembert principle }
\label{section2}
We devote this section to study how some discretizations of the Lagrange-d'Alembert principle lead to discretizations of the nonholonomic DAE \eqref{NhDAE3} and furthermore, under some hypotheses, to discrete schemes approximating the nonholonomic flow and preserving $D$ in the sense of definition \ref{pdisc}. 
First, we introduce a notion of discretization of Hamilton's variational principle (\cite{MarsdenWest,Mose}), leading to variational integrators. Then we focus on similar techniques applied to nonholonomic systems (\cite{CoMa,McPer}), mainly following the approach by McLachlan and Perlmutter \cite{McPer}.

\subsection{Discrete mechanics and variational integrators}

Variational integrators are a kind of geometric integrators for the Euler-Lagrange equations which retain  their variational character and also, as a consequence,  some of the main geometric properties of the continuous
system, such as symplecticity and momentum conservation (see
\cite{Hair,MarsdenWest}).
  In the following we will summarize the main features of this type
of geometric
 integrators.  A {\it discrete Lagrangian} is a map
$L_d:Q\times Q\rightarrow\R$, which may be considered as
an approximation of the action integral defined by a continuous  Lagrangian $L\colon TQ\to
\R$, that is
\begin{equation}\label{aprox}
L_d(q_0, q_1)\approx \int^{t_0+\epsilon}_{t_0} L(q(t), \dot{q}(t))\; dt,
\end{equation}
where $q(t)$ is a solution of the Euler-Lagrange equations for $L$ joining $q(t_0)=q_0$ and $q(t_0+\epsilon)=q_1$ for small enough $\epsilon>0$.

 Define the {\it action sum} $S_d\colon Q^{N+1}\to
\R$  corresponding to the Lagrangian $L_d$ by
\begin{equation}\label{DiscAction}
{S_d}=\sum_{k=1}^{N}  L_d(q_{k-1}, q_{k}),
\end{equation}
where $q_k\in Q$ for $0\leq k\leq N$, $N$ is the number of discretization steps. The discrete variational
principle   states that the solutions of the discrete system
determined by $L_d$ must extremize the action sum given fixed
endpoints $q_0$ and $q_N$. By extremizing ${S_d}$ over $q_k$,
$1\leq k\leq N-1$, we obtain the system of difference equations
\begin{equation}\label{discreteeq}
 D_1L_d( q_k, q_{k+1})+D_2L_d( q_{k-1}, q_{k})=0,
\end{equation}
or, in coordinates,
\[
\frac{\partial L_d}{\partial x^i}(q_k, q_{k+1})+\frac{\partial
L_d}{\partial y^i}(q_{k-1}, q_{k})=0,
\]
where $1\leq i\leq n,\ 1\leq
k\leq N-1$, and $x$ and $y$ represent the $n$ first and $n$ last variables of the function $L_d$, respectively.

These  equations are usually called the  {\it discrete
Euler--Lagrange equations}. Under some regularity hypotheses (the
matrix $D_{12}L_d(q_k, q_{k+1})$ is supposed to be regular), it is possible to
define  from (\ref{discreteeq}) a (local) discrete flow map $ F_{L_d}\colon Q\times
Q\to  Q\times Q$, by $F_{L_d}(q_{k-1}, q_k)=(q_k,
q_{k+1})$. We will refer to the $F_{L_d}$ flow, and also (with some abuse of notation) to the equations \eqref{discreteeq}, as a {\it variational integrator.}

Define the  discrete
Legendre transformations associated to  $L_d$ by
\begin{eqnarray*}
\F L_d^-\colon Q\times Q&\to & T^*Q\\
(q_k, q_{k+1})&\longmapsto& (q_k, -D_1 L_d(q_k, q_{k+1})) ,\\
\F L_d^+\colon  Q\times Q&\to&  T^*Q\\
(q_k, q_{k+1})&\longmapsto& (q_{k+1}, D_2 L_d(q_k, q_{k+1})) ,
\end{eqnarray*}
and the discrete Poincar{\'e}--Cartan 2-form by $\Omega_d=(\F L_d^+)^*\Omega_Q=(\F L_d^{-})^*\Omega_Q$,
where $\Omega_Q$ is the canonical symplectic form on $T^*Q$. The
discrete algorithm determined by $F_{L_d}$ preserves the
symplectic form, i.e., $F_{L_d}^*\Omega_d=\Omega_d$.
Moreover, if the discrete Lagrangian is invariant under the
diagonal action of a Lie group $G$, then the discrete momentum map
$J_d\colon Q\times Q \to  {\mathfrak g}^*$ defined by
\[ \langle
J_d(q_k, q_{k+1}), \xi\rangle=\langle D_2L_d(q_k, q_{k+1}),
\xi_Q(q_{k+1})\rangle \]
is preserved by the discrete flow. Therefore, these integrators are symplectic-momentum preserving. Here, $\xi_Q$ denotes the fundamental vector field
determined by $\xi\in {\mathfrak g}$, where ${\mathfrak g}$ is the Lie algebra of $G$. (See \cite{MarsdenWest} for more details.) Everything developed in this section is true for a general configuration manifold $Q$.

\subsection{Nonholonomic integrators}

Discretizations of the Lagrange-d'A\-lem\-bert principle for Lagrangian systems with nonholonomic cons\-traints have been introduced in \cite{CoMa,McPer}, for instance. Under some regularity conditions, these discretizations allow to construct numerical integrators that approximate the continuous flow fairly well, and can be linearly implicit, semi-implicit or implicit.  

To define the notion of a discrete nonholonomic system providing a discrete flow on a submanifold of $Q\times Q$ and, as well, a corresponding version of discrete Euler-Lagrange equations \eqref{discreteeq}, one needs three ingredients: a discrete Lagrangian, a constraint distribution $D\subset TQ$ and a discrete constraint space $D_d\subset Q\times Q$. In the following, we follow \cite{McPer}. 

\begin{definition}\label{DiscNHSet} 
A discrete nonholonomic system is given by the triple $(Q\times Q,L_d,D_d)$:
\begin{enumerate}
\item $D_d$ is a submanifold of $Q\times Q$ of dimension $2n-m$ with the additional property that
\[
I_d=\lc (q,q)\,|\,q\in Q\rc\subset D_d.
\]
\item $L_d:Q\times Q\Flder\R$ is the discrete Lagrangian.
\end{enumerate}
\end{definition}

We define the discrete Lagrange-d'Alembert principle (DLA) to be the extremization of the action sum in \eqref{DiscAction} among all sequences of points $\lc q_k\rc$ with given fixed end points $q_0,q_N$, where the variations must satisfy $\delta q_k\in D_{q_k}$ (in other words $\delta q_k\in \mbox{ker}\,\mu^{\alpha}$) and $(q_k,q_{k+1})\in D_d$ for all $k\in\lc 0,...,N-1\rc$. This leads to the conditions
\[
\lp D_1L_d(q_k,q_{k+1})+D_2L_d(q_{k-1},q_{k})\rp\,\delta q_k=0,\,\,\,1\leq k\leq N-1,
\]
for all variations $\delta q_k$, where $\bra\mu^{\alpha},\delta q_k\ket=0$ along with $(q_k,q_{k+1})\in D_d.$ This leads to the following set of discrete nonholonomic equations
\begin{subequations}\label{DLA}
\begin{align}
D_1L_d(q_k,q_{k+1})&+D_2L_d(q_{k-1},q_{k})=\lambda_{\alpha}\,\mu^{\alpha}(q_k),\label{DLAa}\\
(q_k,q_{k+1})&\in D_d\label{DLAb}.
\end{align}
\end{subequations}
For the sake of clarity, the condition \eqref{DLAb} may be rewritten as $\phi^{\alpha}_d(q_k,q_{k+1})=0$, where $\phi^{\alpha}_d:Q\times Q\Flder\R$ is the set of $m$ functions whose annihilation defines $D_d$ (in other words, a suitable discretization of the nonholonomic constraints \eqref{LC}). Equations \eqref{DLA}, where $\lambda_{\alpha}=(\lambda_k)_{\alpha}$ is chosen appropriately, define a {\it discrete nonholonomic flow map} $F_{L_d}^{nh}:D_d\Flder D_d$ given by
\begin{equation}\label{NHFlow}
F_{L_d}^{nh}(q_{k-1},q_k)=(q_k,q_{k+1}),
\end{equation} 
where $q_{k+1}$ satisfyes \eqref{DLA} provided that $(q_{k-1},q_k)\in D_d$, if and only if the following regularity condition is fullfiled for each $(q_k,q_{k+1})$ in a neighborhood of the diagonal of $Q\times Q$. To state this regularity condition, we make the definition:

\begin{definition}
For each $(q_{k-1},q_{k})\in Q\times Q$, we define the map $\gamma_{(q_{k-1},q_{k})}:Q\Flder D^*_{q_{k}}$ by
\[
\gamma_{(q_{k-1},q_{k})}(q_{k+1})= i^*_{q_{k}}\lp D_1L_d(q_k,q_{k+1})+D_2L_d(q_{k-1},q_{k})\rp,
\]
where $D^*_{q_{k}}$ is the dual space of $D_{q_{k}}\subset T_{q_{k}}Q$ and $i^*_{q_{k}}:T_{q_{k}}^*Q\Flder D^*_{q_{k}}$ is the map dual to the linear  inclusion $i_{q_{k}}:D_{q_{k}}\hookrightarrow T_{q_{k}}Q$.
\end{definition}

\begin{proposition}\label{REG}
Let $(q_{k-1},q_{k})\in D_d$. Let $\pi_1:Q\times Q\Flder Q$ be the projection on the first factor. Suppose that $\pi_1|_{D_d}:D_d\Flder Q$. The discrete nonholonomic flow map $F_{L_d}^{nh}$ is then guaranteed to exist locally uniquely provided that for each $q_{k+1}\in\gamma^{-1}_{(q_{k-1},q_{k})}(0)\cap(\pi_1|_{D_d})^{-1}(q_{k})$, and for each non-zero $v_{q_{k+1}}\in T_{q_{k+1}}D_d$,
\[
\bra D_2D_1L_d(q_{k},q_{k+1})\cdot v_{q_{k+1}},\,v_{q_{k}}\ket\neq 0,
\]
holds for all $v_{q_{k}}\in D_{q_{k}}$. When this condition holds for all $q_{k}\in Q$, the discrete nonholonomic equations produce a uniquely defined diffeomorphism $F_{L_d}^{nh}:D_d\Flder D_d$ in the way described above.
\end{proposition}
We point out, that this regularity condition has been previously introduced in \cite{CoMa} phrased as follows: the discrete nonholonomic flow map $F_{L_d}^{nh}:D_d\Flder D_d$ is uniquely defined as a local diffeomorphism if and only if the matrix
\begin{equation}\label{RegularityDLA}
\lp
\begin{array}{cc}
D_1D_2L_d(q_k,q_{k+1}) & \lp\mu^{\alpha}(q_k)\rp\\\\
D_2\phi^{\alpha}_d(q_k,q_{k+1}) &0
\end{array}
\rp
\end{equation}
is invertible. Roughly speaking, this regularity condition, when it holds, allows us to determine $\lambda_{\alpha}=(\lambda_k)_{\alpha}$ in \eqref{DLA} in terms of the other variables and therefore to choose it appropriately as mentioned above. The form of $(\lambda_k)_{\alpha}$ depends on the discretization constraints $\phi^{\alpha}_d$ and is quite involved in general. We refer to \S\ref{examples} for particular examples. Again, everything in this section is true for a general configuration manifold $Q$.

We will refer to the algorithm described by \eqref{NHFlow} as a {\it nonholonomic variational integrator} in analogy to the unconstrained case.

\begin{remark}
{\rm
The discrete nonholonomic flow map generates points $(q_k,$ $q_{k+1})$ belonging to $D_d$, from points $(q_{k-1},q_k)\in D_d$. As mentioned before, we look for some conditions guaranteing that  $v_{q_k}\in D_{q_k}$ for any $k$. This is not straightforward from the relations in \eqref{DLA}, since a definition of the pair $(q_k,v_k)$ from $(q_{k-1},q_k)$ is required. This issue will be discussed below. 
}
\end{remark}

\begin{remark}
{\rm
The property that $v_{q_k}\in D_{q_k}$ for any $k$, is automatically achieved if $q_{k+1}$ and $q_k$ are connectable by the time-continuous nonholonomic flow map $F^{nh}_t:D\Flder D$ at time $\epsilon$, i.e., if $F^{nh}_{\epsilon}(v_{q_k})=v_{q_{k+1}}$. In \cite{McPer} the subset of $Q\times Q$ consisiting of pairs $(q_k,q_{k+1})$ such that there exists a $C^2$ curve $q(t)$ joining $q_k$ and $q_{k+1}$ in time $\epsilon$, 
and such that the curve $(q(t),\dot q(t))$ satisfies the Lagrange-d'Alembert equations is called {\it exact discrete constraint distribution} and denoted by $D_d^E$. 
Moreover, it is proved that there exist smooth coordinates realizing $D_d^E$ as a submanifold of $Q\times Q$ of dimension $2n-m$. Here we do not explore this object further, rather we are going to address the issue raised in the previous remark directly. 
}

\end{remark}

\subsection{Construction of nonholonomic variational integrators by means of finite difference maps}
As has been shown, the definition of a discrete nonholonomic system requires the tuple $(Q\times Q,L_d,D_d)$ to be specified. For some reasons (which will become clear in the section devoted to examples), in many cases it is preferable to specify $L_d$ and $D_d$ by means of a so-called {\it finite difference map} $\rho$ (see \cite{McPer} for further details).
\begin{definition}\label{FDM}
A {\rm finite difference map} $\rho$ is a diffeomorphism $\rho:U(I_d)\Flder V(Z)$, where $U(I_d)$ is a neighborhood of the diagonal $I_d$ in $Q\times Q$, and $V(Z)$ denotes a neighborhood of the zero section of $TQ$, i.e $Z:Q\Flder TQ$ s.t. $Z(q)=0_q\in T_qQ$, which satisfies the following conditions:
\begin{enumerate}
\item $\rho(I_d)=Z$,
\item $\tau_Q\circ\rho(U(I_d))=Q$, where $\tau_Q:TQ\Flder Q$ is the canonical projection,
\item $\tau\circ\rho|_{I_d}=\pi_1|_{I_d}=\pi_2|_{I_d}$, where $\pi_i$ are the projections from $Q\times Q$ to $Q$.
\end{enumerate}
\end{definition}
Next, we give a  local definition of $\rho$, i.e. $\rho_{q}:U_{q}\Flder\tilde U_{q}\times\R^n$:
\begin{equation}\label{roloc}
\rho_q(\bar q)=\rho(q,\bar q)=\lp g^i_q(\bar q),f^i_q(\bar q)\rp\in T_{_{g_q(\bar q)}}Q,
\end{equation}  
where $\bar q\in U_q$, and $g^i_{q},f^i_{q}:U_{q}\Flder\R$, with $i=1,...,n$ are smooth functions. In other words, $f^i_q(\bar q)$ are the local coordinates of $v_{_{g_q(\bar q)}}\in T_{_{g_q(\bar q)}}Q$ (recall that, locally, $\lp g_q(\bar q),f_q(\bar q)\rp\in V(Z(q))$). From the definition of $\rho$ in definition \ref{FDM}, these functions must satisfy two conditions, namely
\[
g_q(q)=q\,\,\,\, \mbox{and}\,\,\,\,f_{q}(q)=0.
\]
For the sake of brevity, we set $\tilde q:=g_q(\bar q)\in \tilde U_q$ and $v_{\tilde q}:=f_q(\bar q)\in\R^n$.

The finite difference map $\rho$ is crucial in our construction since it can be used to define both the discrete Lagrangian function and the discrete constraint distribution $D_d$ from the continuous one as stated in the following proposition (recall that $D_d$ is defined by the annihilation of the functions $\phi^{\alpha}_d$).
\begin{proposition}\label{Ff}
Given a diffeomorphism $\rho$, define the functions $\phi_d:U(I_d)\Flder \R^{2n-m}$ by
\[
\phi^{\alpha}_d=\mu^1\circ\rho\times...\times\mu^{2n-m}\circ\rho=:\mu^{\alpha}\circ\rho.
\]
Then, $\phi_d$ is a submersion, so in particular, $0$ is a regular value of $\phi_d$ and consequently {\rm $D_d:=\phi_d^{-1}(0)$} is a well-defined discrete constraint submanifold in the terms of definition \ref{DiscNHSet}.
\end{proposition}
See \cite{McPer} for the proof. Thus, it is sufficient to consider the continuous one-forms $\mu^{\alpha}$ and, in addition, the difference map $\rho$ in order to construct the discrete constrained distribution $D_d$. Notice that we could have chosen any set of $m$ independent functions $\psi^{\alpha}:TQ\Flder\R$ in the previous proposition to define $D_d$, namely $\phi^{\alpha}=\psi^{\alpha}\circ\rho$. However we pick $\mu^{\alpha}:TQ\Flder\R$ (the one-forms defining $D$), a choice that has some advantages as we'll see shortly in proposition \ref{propoDcont}.
Note as well that in the previous proposition, the discrete functions $\phi_d$ are not defined everywhere on $Q\times Q$, but in a neighborhood of its diagonal only.

\begin{definition}\label{VNI}
Consider the commutative diagram
\begin{equation}\label{diagram}
\xymatrix{
(q_{k-1},q_k)\ar[rr]^{F_{L_d}^{nh}}\ar[dd]_{\rho} & & (q_k,q_{k+1})\ar[dd]^{\rho}\\\\
(\tilde q_k,v_{\tilde q_k})\ar[rr]_{\tilde F_{L_d}} & & (\tilde q_{k+1},v_{\tilde q_{k+1}})
} 
\end{equation} 
where $q_{k}\in U_{q_{k-1}}$ and $q_{k+1}\in U_{q_{k}}$. We define the local flow $\tilde F_{L_d}: T_{\tilde q_k}Q\Flder T_{\tilde q_{k+1}}Q$
by
\[
\tilde F_{L_d}:=\rho\circ F_{L_d}^{nh}\circ \rho^{-1},
\]
which we call a {\rm velocity nonholonomic variational integrator}. 
\end{definition}

In the next proposition we establish a sufficient condition for $\tilde F_{L_d}$ to preserve the nonholonomic distribution $D$, in other words $\tilde F_{L_d}:D_{\tilde q_k}\Flder D_{\tilde q_{k+1}}$.
\begin{proposition}\label{propoDcont}
Assume that $(\tilde q_k,v_{\tilde q_k})\in D_{\tilde q_k}$ and that $D_d$ is defined as in proposition \ref{Ff}, that is using $\phi_d^{\alpha}=\mu^{\alpha}\circ\rho$. Then $(\tilde q_{k+1},v_{\tilde q_{k+1}})$ defined by $\tilde F_{L_d}$ belongs to $D_{\tilde q_{k+1}}$. In other words, $\tilde F_{L_d}:D_{\tilde q_k}\Flder D_{\tilde q_{k+1}}$. 
\end{proposition}
\begin{proof}
We are assuming that $(\tilde q_k,v_{\tilde q_k})\in D_{\tilde q_k}$. Considering that $(\tilde q_k,v_{\tilde q_k})$ and $(q_{k-1},q_k)$ are $\rho-$related this inmediately implies that $(q_{k-1},q_k)\in D_d$ since we are assuming that $D_d$ (proposition \ref{Ff}) is defined by the annihilation of  $\phi^{\alpha}_d=\mu^{\alpha}\circ\rho$ (note that it would be enough to consider one-forms proportional to $\mu^{\alpha}$, i.e. $\phi^{\alpha}_d\propto\mu^{\alpha}\circ\rho$). In other words
\[
\mu^{\alpha}(\tilde q_k)v_{\tilde q_k}=0\,\Flder \mu^{\alpha}(g_{q_{k-1}}(q_k))f_{q_{k-1}}(q_k)=0\Flder (q_{k-1},q_k)\in D_d.
\]
By diagram \eqref{diagram}, $(q_k,q_{k+1})\in D_d$, which applying $\rho$ implies
\[
\mu^{\alpha}(g_{q_{k}}(q_{k+1}))f_{q_{k}}(q_{k+1})=\mu^{\alpha}(\tilde q_{k+1})v_{\tilde q_{k+1}}=0,
\]
which implies that $(\tilde q_{k+1},v_{\tilde q_{k+1}})\in D_{\tilde q_{k+1}}$ as claimed.
\end{proof}


\begin{remark}\label{question}
{\rm
The previous assertion includes a redefinition of the original nodes, that is $\lc q_k\rc\Flder \lc \tilde q_k\rc$, and suggests a particular choice of the discrete constraints $\phi^{\alpha}_d$. In general, the discrete nonholonomic variational integrator $F_{L_d}^{nh}$ will not define, through $\rho$, a velocity nonholonomic integrator $\tilde F_{L_d}$ which preserves the nonholonomic distribution $D$ at the original nodes. 
Thus, the question arises, whether this is true for a perturbed constraint manifold. We are going to discuss this question in case of the nonholonomic particle example below.
}
\end{remark}

\begin{remark}
{\rm
Given a complete sequence $\lc q_0,q_1,q_2,...,q_{N-1},q_N\rc$, the generation of new nodes by the finite difference map requires a choice for the end points, i.e. either $\tilde q_0$ or $\tilde q_N$. This can be seen easily in the following way: In general we will choose the new nodes to be $\tilde q_{k+1}=g_{q_k}(q_{k+1})$, which defines univocally the sequence $\lc\tilde q_1,...,\tilde q_N\rc$; therefore $\tilde q_0$ remains undefined and we have to choose it independently, a natural choice being  $\tilde q_0=q_0$. Similarly, 
if we define the new nodes to be  $\tilde q_{k}=g_{q_k}(q_{k+1})$, then $\tilde q_N$ is the undefined end point, a natural choice being $\tilde q_N=q_N$. The same holds true for the velocities $v_{\tilde q}=f_{q}(\bar q).$
}
\end{remark}
\begin{remark}
{\rm
Notice that the discretization of the nonholonomic problem established in this section provides a set of discrete multipliers besides of the discrete dynamical variables $(q_k,v_{q_k})$. More concretely, the regularity condition \eqref{RegularityDLA} ensures that $\lambda_{\alpha}$ can be worked out in terms of $(q_{k-1},q_k,q_{k+1})$ and therefore, through the finite difference map, in terms of $(q_k,v_{q_k}).$ The accuracy of these discrete multipliers with respect to the continuous evolution $\lambda(t)=\lambda(q(t),v(t))$ is determined by comparison to \eqref{lambdax}.
}
\end{remark}
As mentioned before, $\rho$ will also be used to define the discrete Lagrangian $L_d:Q\times Q\Flder\R$, accomplishing the task of defining the discrete nonholonomic system $(Q\times Q,L_d,D_d)$, particularly by
\[
L_d=\epsilon\, L\circ\rho,
\]
where $L$ denotes the continuous Lagrangian and $\epsilon$ the step size of the discretization. 

As before, all the considerations in this section make sense for a general configuration manifold $Q$. In the next section we shall restrict again to $\R^n$ by considering two particular examples of $\rho$ parametrized by $\beta$, namely
\begin{equation}\label{robeta}
\rho^{\beta}(q_k,q_{k+1})=\lp(1-\beta)\,q_k+\beta\,q_{k+1}\,,\,\frac{q_{k+1}-q_k}{\epsilon}\rp,
\end{equation}
corresponding to $g_q(\bar q)=(1-\beta)\,q+\beta\,\bar q$ and $f_q(\bar q)=\frac{\bar q-q}{\epsilon}$ in \eqref{roloc}. These $\rho^{\beta}$ lead to the following family of discrete Lagrangians
\[
L_d^{\beta}(q_k,q_{k+1})=\epsilon\,L\lp(1-\beta)\,q_k+\beta\,q_{k+1}\,,\,\frac{q_{k+1}-q_k}{\epsilon}\rp. 
\]
As pointed out in \eqref{aprox}, discrete Lagrangians are approximations of the action integral over time intervals of length $\epsilon$. Next, using usual arguments based on Taylor expansions, we are able to determine the order of consistency of $L_d^{\beta}(q_k,q_{k+1})$ with respect to $\int_0^{\epsilon}L(q(t),\dot q(t))dt$ (see \cite{MarsdenWest} for more details):

\[
\int_0^{\epsilon}L(q(t),\dot q(t))\,dt=\epsilon\,L(q,\dot q)+\frac{\epsilon^2}{2}\,\lp\frac{\der L}{\der q}\,\dot q+\frac{\der L}{\der\dot q}\,\ddot q\rp+O(\epsilon^3),
\]
where $q(0)=q$ and $\dot q(0)=\dot q$. On the other hand, recalling that $q_k\simeq q(0)$ and $q_{k+1}\simeq q(\epsilon)$, we find that
\[
L_d^{\beta}(q_k,q_{k+1})=\epsilon\,L(q,\dot q)+\frac{\epsilon^2}{2}\,\lp 2\beta\,\frac{\der L}{\der q}\,\dot q+\frac{\der L}{\der\dot q}\,\ddot q\rp+O(\epsilon^3).
\]
Therefore, if we consider $\beta\in [0,1]$, $L_d^{\beta}(q_k,q_{k+1})$ is second-order consistent with respect to the action integral if $\beta=1/2$, otherwise it is only consistent of order one. In the examples we shall consider the cases $\beta=0$ and $\beta=1/2$.

\section{Examples}\label{examples}
In this section we focus on the particular class of {\it simple mechanical Lagrangians} when $Q=\R^n$. We construct nonholonomic variational integrators according to the procedure shown in the previous section and study their consistency properties with respect to the continuous dynamics. Furthermore, we focus on the example of the nonholonomic particle, studying carefully how particular discretizations generate numerical schemes that preserve the nonholonomic distribution $D$, and to which extent $D$ is perturbed otherwise.

\subsection{The class of simple mechanical Lagrangians}
Assume that the Lagrangian function $L:TQ\Flder\R$ is of the form $L=T-V$, where $T$ is the kinetic energy associated with the Euclidean metric on $\R^n$ and $V$ is the potential energy, i.e. 
\[
L(q,\dot q)=\frac{1}{2}\,\dot q^TM\dot q-V(q),
\]
where $M$ is the $n\times n$ mass matrix. The equations  \eqref{NhDAE3} become
\begin{subequations}\label{DAEMec}
\begin{align}
\dot q&=v,\\
\dot v&=-M^{-1}\nabla_qV(q)+\lambda_{\alpha}\,M^{-1}\,\mu^{\alpha}(q),\\
0&=\mu^{\alpha}(q)\,v,
\end{align} 
\end{subequations}
where $M^{-1}\nabla_qV(q)=(M^{-1})^{ij}\frac{\der V(q)}{\der q^j}$, $\lambda_{\alpha}\,M^{-1}\,\mu^{\alpha}(q)=\lambda_{\alpha}\,(M^{-1})^{ij}\,\mu^{\alpha}_j(q)$ and, of course, $\mu^{\alpha}(q)\,v=\mu^{\alpha}_i(q)\,v^i$. According to proposition \ref{h}, equations \eqref{DAEMec} induce an ODE on $D$ given by
\begin{subequations}\label{ODExMec}
\begin{align}
\dot q&= v,\label{ODExMeca}\\
\dot v&=-M^{-1}\nabla_qV(q)+\lambda_{\alpha}(q,v)\,M^{-1}\,\mu^{\alpha}(q),\label{ODExMecb}
\end{align}
\end{subequations}
where 
\begin{equation}\label{lambdaQV}
\lambda_{\alpha}(q,v)=-C_{\alpha\beta}(q)\lp v\,\frac{\der\mu^{\beta}(q)}{\der q}\,v-\mu^{\beta}(q)M^{-1}\nabla_qV(q)\rp 
\end{equation}
according to \eqref{lambdax} ($v\,\frac{\der\mu^{\beta}(q)}{\der q}\,v=\frac{\der\mu^{\beta}_i(q)}{\der q^j}\,v^i\,v^j$ and $\mu^{\beta}(q)M^{-1}\nabla_qV(q)=$ \\$\mu^{\beta}_i(q)(M^{-1})^{ij}\frac{\der V(q)}{\der q^j}$), and such that $\mu^{\alpha}(q)\,v=0$. In this particular case, $\lp C_{\alpha\beta}\rp$ is the inverse matrix of $\lp C^{\alpha\beta}\rp=\lp \mu^{\alpha}(q)M^{-1}\mu^{\beta}(q)\rp$. Employing the coordinates \eqref{Embedd}, the ODE \eqref{ODExi} reads
\begin{subequations}\label{ODExiConnec}
\begin{align}
\dot q^a&=v^a,\label{ODExiConneca}\\
\dot q^{\alpha}&=-A^{\alpha}_a(q)\,v^a,\label{ODExiConnecb}\\
\dot v^a&=-M^{aj}\nabla_{q^j}V(q)+\lambda_{\alpha}(q,v)\lp M^{ab}A^{\alpha}_b(q)+M^{a\alpha}\rp,\label{ODExiConnecc}
\end{align}
\end{subequations}
where we use the Ehresmann connection $\mu^{\alpha}_i(q)=(A^{\alpha}_a(q),\delta^{\alpha}_{\beta})$. In addition, 

\begin{multline*}
\lambda_{\alpha}(q,v)=-C_{\alpha\beta}(q)\left( v^b\frac{\der A^{\beta}_b(q)}{\der q^a}v^a-v^b\frac{\der A^{\beta}_b(q)}{\der q^{\gamma}}A^{\gamma}_a(q)\,v^a \,\,\,\,\,\right.\\
\,\,\left.-A^{\beta}_a(q)M^{aj}\nabla_{q^j}V(q)-M^{\beta j}\nabla_{q^j}V(q)\right).
\end{multline*}

\subsection{$(1,0)$ linearly explicit integrator} 
Consider the finite difference map
\begin{equation}\label{ro1}
\rho(q_k,q_{k+1})=\lp q_k,\frac{q_{k+1}-q_k}{\epsilon}\rp,
\end{equation}
which corresponds to the case $\beta=0$ in \eqref{robeta}. Consider the discrete Lagrangian $L_d:Q\times Q\Flder\R$ generated by $\rho$:
\begin{equation}\label{Ld1}
L_d(q_k,q_{k+1})=\frac{1}{2\epsilon}(q_{k+1}-q_k)^T\,M\,(q_{k+1}-q_k)-\epsilon\,V(q_k).
\end{equation}
It has been proved above that this discrete Lagrangian gives rise to a first order consistent approximation of the action integral over a time interval of lenght $\epsilon$; therefore, in the spirit of \cite{MarsdenWest}, we expect a first order consistent integrator for the nonholonomic flow. The nonholonomic integrator \eqref{DLA} is given by 
\begin{eqnarray*}
\frac{q_{k+1}-2q_k+q_{k-1}}{\epsilon^2}+M^{-1}\,\nabla_qV(q_k)&=&(\lambda_{k+1})_{\alpha}M^{-1}\,\mu^{\alpha}(q_k),\\
\mu^{\alpha}(q_k)\,\frac{q_{k+1}-q_k}{\epsilon}&=&0,
\end{eqnarray*}
where the rescaling $\lambda_{k+1}\mapsto-\lambda_{k+1}/\epsilon$ has been performed (equivalently this accounts for the choice $\phi^{\alpha}_d=-\epsilon\,\,\mu^{\alpha}\circ\rho$, which keeps us in the case $\phi^{\alpha}_d\propto \mu^{\alpha}\circ\rho$ and the conditions established by proposition \ref{propoDcont}). Considering the particular form of $\rho$, a velocity nonholonomic integrator is defined by $\tilde q_{k+1}:=q_k$ and $v_{\tilde q_{k+1}}:=v_{k+1}:=(q_{k+1}-q_k)/\epsilon$, which leads to the velocity formulation $\tilde F_{L_d}$:
\begin{subequations}\label{DAELd1}
\begin{align}
\tilde  q_{k+1}&=\tilde q_k+\epsilon\,v_k,\label{DAELd1a}\\
v_{k+1}&=v_k-\epsilon\,M^{-1}\nabla_qV(\tilde q_{k+1})+\epsilon\,(\lambda_{k+1})_{\alpha}M^{-1}\mu^{\alpha}(\tilde q_{k+1}),\label{DAELd1b}\\
\mu^{\alpha}(\tilde q_{k+1})\,v_{k+1}&=0,\label{DAELd1c}
\end{align}
\end{subequations}
where the initial condition should satisfy the constraint $\mu^{\alpha}(\tilde q_k)\,v_k=0$. As expected, this integrator respects the nonholonomic constraint described by $D$, which is ensured by the equation \eqref{DAELd1c}. Using usual arguments based on Taylor expansions, one can check that the previous integrator is first-order with respect to the time-continuous equations \eqref{DAEMec}. The discrete Lagrange multipliers $\lambda_{k+1}$ can be determined by inserting \eqref{DAELd1b} into \eqref{DAELd1c} and solving the linear system
\[
(\lambda_{k+1})_{\alpha}\,\mu^{\alpha}(\tilde q_{k+1})\,M^{-1}\,\mu^{\beta}(\tilde q_{k+1})=\mu^{\beta}(\tilde q_{k+1})\lp M^{-1}\,\nabla_qV(\tilde q_{k+1})-v_k/\epsilon\rp,
\]
i.e. 
\begin{equation}\label{LamLd1}
(\lambda_{k+1})_{\alpha}=-\frac{1}{\epsilon}\,C_{\alpha\beta}(\tilde q_{k+1})\mu^{\beta}(\tilde q_{k+1})\lp v_k-\epsilon\,M^{-1}\,\nabla_qV(\tilde q_{k+1})\rp
\end{equation}
which makes the method linearly explicit (note that $\lambda_{k+1}$ depends on $v_k$ and $\tilde q_{k+1}$, which, at the same time, depends on $\tilde q_k$ and $v_k$; therefore the right hand side of \eqref{DAELd1a} and \eqref{DAELd1b} depends explicitly on $\tilde q_k$ and $v_k$). 
\begin{proposition}\label{firstorder}
The velocity nonholonomic integrator represented by the scheme \eqref{DAELd1} and \eqref{LamLd1} (determined by the finite difference map \eqref{ro1}) is a $(1,0)$ consistent discretization (in the sense of definition \ref{pdisc}) of the nonholonomic flow given by the $ODE$ in \eqref{ODExMec}.
\end{proposition}
\begin{proof}
For the sake of simplicity we set $V=0$. The strategy of our proof is to use Taylor expansions to show that $|q(t_k+\epsilon)-\tilde q_{k+1}|\sim\, O(\epsilon^2)$ and $|v(t_k+\epsilon)-v_{k+1}|\sim\, O(\epsilon^2)$, where the continuous functions $q(t),\,v(t)$ are determined by \eqref{ODExMec}, while the discrete quantities $\tilde q_{k+1},\,v_{k+1}$ are determined by \eqref{DAELd1}. Since the assertion is easy to prove for $q(t)$ and $\tilde q_{k+1}$, we focus on the $v$ part.

Taking into account that $\tilde q_{k+1}=\tilde q_k+\epsilon\,v_k$ and expanding $\mu^{\beta}(\tilde q_{k+1})$ in \eqref{LamLd1} around $\tilde q_k$ we arrive at
\[
\begin{split}
(\lambda_{k+1})_{\alpha}=&-\frac{1}{\epsilon}C_{\alpha\beta}(\tilde q_{k+1})\left(\mu^{\beta}(\tilde q_{k})\,v_k+\epsilon\,v_k\frac{\der\mu^{\beta}(\tilde q_k)}{\der q}\,v_k\right.\\&\left.-\epsilon\,\mu^{\beta}(\tilde q_{k+1})\,M^{-1}\,\nabla_qV(\tilde q_{k+1})+O(\epsilon^2)\right).
\end{split}
\]
Recall that the method \eqref{DAELd1} respects the nonholonomic constraint \eqref{DAELd1c}, i.e., \\ 
$\mu^{\alpha}(\tilde q_k)\,v_k=0$. Inserting the previous expression for $(\lambda_{k+1})_{\alpha}$ into  \eqref{DAELd1b} and expanding around $\tilde q_k$ we obtain
\begin{equation}\label{vExpan}
v_{k+1}=v_k-\epsilon\,C_{\alpha\beta}(\tilde q_k)\lp v_k\frac{\der\mu^{\beta}(\tilde q_k)}{\der q}\,v_k\rp\,M^{-1}\mu^{\alpha}(\tilde q_k)+O(\epsilon^2).
\end{equation}
On the other hand, $v(t_k+\epsilon)=v(t_k)+\epsilon\,\dot v(t_k)+O(\epsilon^2)$ or, taking into account that $v(t_k)\simeq v_k$, $v(t_k+\epsilon)=v_k+\epsilon\,\dot v_k+O(\epsilon^2)$. Using the equations \eqref{ODExMecb} and \eqref{lambdaQV}, we find out that $v(t_k+\epsilon)$ is equal to \eqref{vExpan} up to $O(\epsilon)$-terms. Thus, $|v(t_k+\epsilon)-v_{k+1}|\sim\, O(\epsilon^2)$ and therefore the $p=1$ order follows.

Now, we show the order of accuracy of the discrete Lagrange multipliers \eqref{LamLd1} with respect to the continuous expression \eqref{lambdaQV}. In particular,  \eqref{LamLd1} leads to
\begin{eqnarray*}
&&(\lambda_{k+1})_{\alpha}=-C_{\alpha\beta}(\tilde q_k)\lp v_k\frac{\der\mu^{\beta}(\tilde q_k)}{\der q}\,v_k-\mu^{\beta}(\tilde q_{k})M^{-1}\,\nabla_qV(\tilde q_{k})\rp\\
&&+\epsilon\,C_{\alpha\beta}(\tilde q_k)\lp\nabla_q\mu^{\beta}(\tilde q_k)v_kM^{-1}\nabla_qV(\tilde q_k)+\mu^{\beta}(\tilde q_k)M^{-1}\Delta V(\tilde q_k)v_k-\frac{1}{2}v_k\Delta \mu^{\beta}(\tilde q_k)v_k^2\rp\\
&&+\epsilon\,\nabla_qC_{\alpha\beta}(\tilde q_k)v_k\lp \mu^{\beta}(\tilde q_k)M^{-1}\nabla_qV(\tilde q_k)-v_k\frac{\der\mu^{\beta}(\tilde q_k)}{\der q}\,v_k\rp,
\end{eqnarray*}
while the continuous expression reads
\[
\lambda(t_k+\epsilon)=\lambda(q(t_k+\epsilon),v(t_k+\epsilon))=\lambda(q(t_k),v(t_k))+\epsilon\,\frac{\der\lambda}{\der q^i}\,\dot q^i(t_k)+\epsilon\,\frac{\der\lambda}{\der v^i}\,\dot v^i(t_k)+O(\epsilon^2),
\]
where
\begin{equation}\label{LamExp}
\begin{split}
&\frac{\der\lambda_{\alpha}}{\der q}\dot q=\nabla_qC_{\alpha\beta}(q)v\lp\mu^{\beta}(q)M^{-1}\nabla_qV(q)-v\frac{\der\mu^{\beta}(q)}{\der q}\,v\rp\\
&+C_{\alpha\beta}(q)\lp\nabla_q\mu^{\beta}(q)vM^{-1}\nabla_qV(q)+\mu^{\beta}(q)M^{-1}\Delta V(q)v-v\Delta\mu^{\beta}(q)v^2\rp,\\\\
&\frac{\der\lambda_{\alpha}}{\der v}\dot v=-2C_{\alpha\beta}(q)\lp\nabla_q\mu^{\beta}(q)v\rp\lp-M^{-1}\nabla_qV(q)+\lambda_{\gamma}(q,v)M^{-1}\mu^{\gamma}(q)\rp,
\end{split}
\end{equation}
$\lambda_{\gamma}(q,v)$ is again determined by \eqref{lambdaQV}. Finally
\[
\begin{split}
|\lambda_{\alpha}(t_k+\epsilon)-&\lp\lambda_{k+1}\rp_{\alpha}|=-\frac{1}{2}\epsilon\,C_{\alpha\beta}(q)\,v\Delta\mu^{\beta}(q)v^2\\&+2\epsilon\,C_{\alpha\beta}(q)\lp\nabla_q\mu^{\beta}(q)v\rp\lp M^{-1}\nabla_qV(q)-\lambda_{\gamma}(q,v)M^{-1}\mu^{\gamma}(q)\rp+O(\epsilon^2).
\end{split}
\]
According to this, the discrete Lagrange multipliers are consistent with the continuous evolution, i.e. $s=0$, and the claim holds.
\end{proof}
Employing the Ehresmann connection, \eqref{DAELd1c} becomes
\[
v_{k+1}^{\alpha}+A^{\alpha}_a(\tilde q_{k+1})\,v^a_{k+1}=0,
\]
for an initial condition satisfying $v_{k}^{\alpha}+A^{\alpha}_a(\tilde q_{k})\,v^a_{k}=0$. Moreover, the equations \eqref{DAELd1a} and \eqref{DAELd1b} can be decomposed into 
\begin{eqnarray*}
\tilde q_{k+1}^a&=&\tilde q_k^q+\epsilon\,v_k^a,\\
\tilde q_{k+1}^{\alpha}&=&\tilde q_k^{\alpha}-\epsilon\,A_a^{\alpha}(\tilde q_k)\,v_k^{a},\\
v_{k+1}^a&=&v_k^a-\epsilon\,M^{aj}\nabla_{q^j}V(\tilde q_k)-\epsilon\,(\lambda_{k+1})_{\alpha}\lp M^{ab}A^{\alpha}_b(\tilde q_k)+M^{a\alpha}\rp+O(\epsilon^2),
\end{eqnarray*}
where
\begin{multline*}
(\lambda_{k+1})_{\alpha}=-C_{\alpha\beta}(\tilde q_k)\left( v^b_k\frac{\der A^{\beta}_b(\tilde q_k)}{\der q^a}v^a_k-v^b_k\frac{\der A^{\beta}_b(\tilde q_k)}{\der q^{\gamma}}A^{\gamma}_a(\tilde q_k)\,v^a_k \,\,\,\,\,\right.\\
\,\,\left.-A^{\beta}_a(\tilde q_k)M^{aj}\nabla_{q^j}V(\tilde q_k)-M^{\beta j}\nabla_{q^j}V(\tilde q_k)\right).
\end{multline*}

\subsection{$(2,0)$ implicit integrator}

A second-order method can be constructed using the following finite difference map
\begin{equation}\label{ro2}
\rho(q_k,q_{k+1})=\lp \frac{q_k+q_{k+1}}{2},\frac{q_{k+1}-q_k}{\epsilon}\rp,
\end{equation}
corresponding to the case $\beta=1/2$ in \eqref{robeta}. The corresponding discrete Lagrangian is given by
\[
L_d(q_k,q_{k+1})=\frac{1}{2\epsilon}(q_{k+1}-q_k)^T\,M\,(q_{k+1}-q_k)-{\epsilon}V(\frac{q_k+q_{k+1}}{2}).
\]
It has been proved above that this discrete Lagrangian implies a second-order consistent aproximation of the action integral; therefore we expect a second-order consistent integrator for the nonholonomic flow. The nonholonomic integrator \eqref{DLA} (the same rescaling as in the above example is performed, i.e. $\phi^{\alpha}_d= -\epsilon\,\mu^{\alpha}\circ\rho$) is
\begin{eqnarray*}
&&\frac{q_{k+1}-2q_k+q_{k-1}}{\epsilon^2}+\frac{1}{2}M^{-1}\,\nabla_qV\lp\frac{q_k+q_{k-1}}{2}\rp\\
&&+\frac{1}{2}M^{-1}\,\nabla_qV\lp\frac{q_k+q_{k+1}}{2}\rp=(\lambda_{k+1})_{\alpha}M^{-1}\,\mu^{\alpha}\lp q_k\rp,\\\\
&&\mu^{\alpha}\lp\frac{q_k+q_{k+1}}{2}\rp\,\frac{q_{k+1}-q_k}{\epsilon}=0.
\end{eqnarray*}
Obviously, this integrator does not preserve the nonholonomic constraint $D$ at the original nodes, neither if we choose $v_k:=(q_{k+1}-q_k)/\epsilon$ nor if we choose $v_{k+1}:=(q_{k+1}-q_k)/\epsilon$. To get the velocity formulation $\tilde F_{L_d}$ we define $\tilde q_{k+1}:=(q_k+q_{k+1})/2$ and $v_{\tilde q_{k+1}}:=v_{k+1}:=(q_{k+1}-q_{k})/\epsilon$ according to the finite difference map $\rho$.  The numerical scheme $(\tilde q_k,v_k)\Flder (\tilde q_{k+1},v_{k+1})$ then reads
\begin{subequations}\label{DAELd2}
\begin{align}
&\tilde q_{k+1/2}=\tilde q_k +\frac{1}{2}\epsilon\,v_{k}\label{DAELd2a}\\
&\tilde q_{k+1}=\tilde q_{k+1/2}+ \frac{1}{2}\epsilon\,v_{k+1},\label{DAELd2b}\\
&v_{k+1}=v_k-\frac{1}{2}\epsilon\,M^{-1}\lp\nabla_qV(\tilde q_{k+1}) +\nabla_qV(\tilde q_{k})\rp+(\lambda_{k+1})_{\alpha}M^{-1}\mu^{\alpha}(\tilde q_{k+1/2}),\label{DAELd2c}\\
&\mu^{\alpha}(\tilde q_{k+1})\,v_{k+1}=0,\label{DAELd2d}
\end{align}
\end{subequations}
where the initial conditions should satisfy the constraint $\mu^{\alpha}(\tilde q_k)\,v_k=0$. At the new nodes $\lc\tilde q_k\rc$ this integrator preserves the nonholonomic constraint as it is ensured by \eqref{DAELd2d}. The method is fully implicit, but it becomes explicit either if $V=0$ or $\nabla_qV=0$. In addition, in the absence of constraints ($\lambda_{k+1}=0$) it reduces to the standard trapezoidal rule. We use the same techniques as in proposition \ref{firstorder} to show that the numerical scheme \eqref{DAELd2} is a second-order integrator of \eqref{ODExMec}.
We set $V=0$ and insert \eqref{DAELd2b} and \eqref{DAELd2c} into \eqref{DAELd2d}; then, by Taylor expansion up to order $\epsilon^2$, we obtain
\[
\lp\mu^{\alpha}(\tilde q_{k+1/2})+\frac{\epsilon}{2}v_{k+1}\nabla_q\mu^{\alpha}(\tilde q_{k+1/2})+O(\epsilon^2)\rp\lp v_k+\epsilon\,(\lambda_{k+1})_{\beta}\,M^{-1}\,\mu^{\beta}(\tilde q_{k+1/2})\rp=0.
\]
Expanding this expression we get  
\begin{eqnarray}
&&(\lambda_{k+1})_{\alpha}=-\lp C_{\alpha\beta}(\tilde q_{k+1/2})+\epsilon\,R_{\alpha\beta}(\tilde q_{k+1/2})\rp\lp v_k\nabla_q\mu^{\beta}(\tilde q_k) v_k+\frac{\epsilon}{2}v_k^2\Delta\mu^{\beta}(\tilde q_k)v_k\right.\nonumber\\
&&\left. \frac{\epsilon}{2}(\lambda_{k+1})_{\gamma}M^{-1}\mu^{\gamma}(\tilde q_{k+1/2})\nabla_q\mu^{\beta}(\tilde q_{k+1/2})v_k +O(\epsilon^2)\rp),\label{lamdaGR}
\end{eqnarray}
where $\nabla_q\mu^{\alpha}=\frac{\der\mu^{\alpha}_i}{\der q^j}$, $\Delta_q\mu^{\alpha}=\frac{\der^2\mu^{\alpha}_i}{\der q^j\der q^k}$ and $v_k^2\,\Delta_q\mu^{\alpha}(\tilde q_k)\,v_k=\frac{\der^2\mu^{\alpha}_i}{\der q^j\der q^l}\,v_k^iv_k^jv_k^l.$ Moreover, we define
\[
R_{\alpha\beta}(\tilde q_{k+1/2}):=-\frac{1}{2}C_{\alpha\gamma}(\tilde q_{k+1/2})\,\lp v_{k+1}\nabla_q\mu^{\gamma}(\tilde q_{k+1/2})M^{-1}\mu^{\rho}(\tilde q_{k+1/2})\rp\,C_{\rho\beta}(\tilde q_{k+1/2}). 
\]
In addition, here we have taken into account that $\mu^{\alpha}(\tilde q_k)\,v_k=0$.
\begin{proposition}\label{secondorder}
The velocity nonholonomic variational integrator given by the scheme \eqref{DAELd2} and \eqref{lamdaGR} (determined by the finite difference map \eqref{ro2}) is a $(2,0)$ consistent discretization (in the sense of definition \ref{pdisc}) of the nonholonomic flow given by the $ODE$ in \eqref{ODExMec}.
\end{proposition}
\begin{proof}
First, we prove that $|q(t_k+\epsilon)-\tilde q_{k+1}|\sim O(\epsilon^3)$. Taking into account that $q(t_k)\sim \tilde q_k$, we have $q(t_k+\epsilon)=\tilde q_k+\epsilon v_k+(\epsilon^2/2) \dot v_k+O(\epsilon^3)$. From \eqref{ODExMecb} we get that $\dot v_k=-C_{\alpha\beta}(\tilde q_k)\lp v_k\,\nabla_q\mu^{\beta}(\tilde q_k)\,v_k\rp M^{-1}\mu^{\alpha}(\tilde q_k)$ yielding
\[
q(t_k+\epsilon)=\tilde q_k+\epsilon v_k-\frac{\epsilon^2}{2}C_{\alpha\beta}(\tilde q_k)\lp v_k\,\nabla_q\mu^{\beta}(\tilde q_k)\,v_k\rp M^{-1}\mu^{\alpha}(\tilde q_k)+O(\epsilon^3).
\]
On the other hand, inserting \eqref{DAELd2a} into \eqref{DAELd2b} we get $\tilde q_{k+1}=\tilde q_k+\frac{\epsilon}{2}v_k+\frac{\epsilon}{2}v_{k+1}$, which, taking into account \eqref{DAELd2c} and the order one terms in \eqref{lamdaGR}, leads to
\[
\tilde q_{k+1}=\tilde q_k+\epsilon v_k-\frac{\epsilon^2}{2}C_{\alpha\beta}(\tilde q_k)\lp v_k\,\nabla_q\mu^{\beta}(\tilde q_k)\,v_k\rp M^{-1}\mu^{\alpha}(\tilde q_k)+O(\epsilon^3),
\]
proving the claim. (Note that the order one terms in \eqref{lamdaGR} are represented by
\[
(\lambda_{k+1})_{\alpha}^0=-C_{\alpha\beta}(\tilde q_k)\lp v_k\nabla_q\mu^{\beta}(\tilde q_k) v_k\rp
\]
which follows by Taylor expansion of all the functions depending on $\tilde q_{k+1/2}$
).

Next, we prove that $|v(t_k+\epsilon)-v_{k+1}|\sim O(\epsilon^3)$. For that purpose, we use $v(t_k+\epsilon)=v_k+\epsilon\dot v_k+(\epsilon^2/2)\ddot v_k+O(\epsilon^3)$ where, from \eqref{ODExMecb} and \eqref{lambdaQV}, we have that 
\begin{eqnarray}
\ddot v=&-&(\nabla_qC_{\alpha\beta}(q)\,v)\lp v\,\nabla_q\mu^{\beta}\,v\rp\,M^{-1}\,\mu^{\alpha}(q)\nonumber\\
&-&C_{\alpha\beta}(q)\lp\dot v\,\nabla_q\mu^{\beta}(q)\,v+v\,\nabla_q\mu^{\beta}(q)\,\dot v+\Delta_q\mu^{\beta}(q)\,v^3\rp\,M^{-1}\mu^{\alpha}(q)\label{ddotv}\\
&-&C_{\alpha\beta}(q)\lp v\,\nabla_q\mu^{\beta}(q)\,v\rp\,M^{-1}(\nabla_q\mu^{\alpha}(q)\,v)\nonumber,
\end{eqnarray}
$\Delta_q\mu^{\beta}(q)\,v^3=\frac{\der^2\mu^{\beta}_i}{\der q^j\der q^l}\,v^i\,v^j\,v^l$. On the other hand, inserting \eqref{lamdaGR} into \eqref{DAELd2c} we get
\begin{eqnarray}
&&v_{k+1}=v_k-\epsilon\,\lp C_{\alpha\beta}(\tilde q_{k+1/2})+\epsilon\,R_{\alpha\beta}(\tilde q_{k+1/2})\rp\lp v_k\nabla_q\mu^{\beta}(\tilde q_k) v_k+\frac{\epsilon}{2}v_k^2\Delta\mu^{\beta}(\tilde q_k)v_k\right.\nonumber\\
&&\left. \frac{\epsilon}{2}(\lambda_{k+1})_{\gamma}M^{-1}\mu^{\gamma}(\tilde q_{k+1/2})\nabla_q\mu^{\beta}(\tilde q_{k+1/2})v_k+O(\epsilon^2)\rp\,M^{-1}\,\mu^{\alpha}(\tilde q_{k+1/2}).\nonumber
\end{eqnarray}
Now, taking into account that $v_{k+1}=v_k+\epsilon\,(\lambda_{k+1})_{\alpha}M^{-1}\mu^{\alpha}(\tilde q_{k+1/2})$, $\tilde q_{k+1/2}=\tilde q_k+\frac{1}{2}\epsilon\,v_k$ and using the expression for $R_{\alpha\beta}(\tilde q_{k+1/2})$ as stated above, we can perform a Taylor expansion of the previous expression for $v_{k+1}$ to obtain
\begin{eqnarray*}
v_{k+1}&=&v_k-\epsilon\,C_{\alpha\beta}(\tilde q_{k})\lp v_k\,\nabla_q\mu^{\beta}(\tilde q_k)\,v_k\rp\,M^{-1}\mu^{\alpha}(\tilde q_k)\\
&-&\frac{\epsilon^2}{2}(\nabla_qC_{\alpha\beta}(\tilde q_k)\,v_k)\lp v_k\,\nabla_q\mu^{\beta}(\tilde q_k)\,v_k\rp\,M^{-1}\,\mu^{\alpha}(\tilde q_k)\\
&-&\frac{\epsilon^2}{2}C_{\alpha\beta}(\tilde q_k)\lp v_k\,\nabla_q\mu^{\beta}(\tilde q_k)(\lambda_{k+1})_{\gamma}M^{-1}\mu^{\gamma}(\tilde q_k)\rp\,M^{-1}\mu^{\alpha}(\tilde q_k)\\
&-&\frac{\epsilon^2}{2}C_{\alpha\beta}(\tilde q_k)\lp (\lambda_{k+1})_{\gamma}M^{-1}\mu^{\gamma}(\tilde q_k)\,\nabla_q\mu^{\beta}(\tilde q_k)\,v_k\rp\,M^{-1}\mu^{\alpha}(\tilde q_k)\\
&-&\frac{\epsilon^2}{2}C_{\alpha\beta}(\tilde q_k)\lp\, v_k^2\Delta_q\mu^{\beta}(\tilde q_k)\,v_k\rp\,M^{-1}\mu^{\alpha}(\tilde q_k)\\
&-&\frac{\epsilon^2}{2}C_{\alpha\beta}(\tilde q_k)\lp v_k\,\nabla_q\mu^{\beta}(\tilde q_k)\,v_k\rp\,M^{-1}(\nabla_q\mu^{\alpha}(\tilde q_k)\,v_k)\\
&+&O(\epsilon^3).
\end{eqnarray*}
As mentioned before,  $(\lambda_{k+1})_{\alpha}=C_{\alpha\beta}(q_k)\lp v_k\nabla_q\mu^{\beta}(q_k)\,v_k\rp+O(\epsilon)$ (which comes from \eqref{lamdaGR}), thus $\dot v_k=-C_{\alpha\beta}(q_k)\lp v_k\nabla_q\mu^{\beta}(q_k)\,v_k\rp\,M^{-1}\mu^{\alpha}(q_k)+O(\epsilon)$. Therefore, comparing the last expression for $v_{k+1}$ and the equation in \eqref{ddotv} we find out that $|v(t_k+\epsilon)-v_{k+1}|\sim O(\epsilon^3)$, establishing the $p=2$ order of consistency as claimed.

Now, regarding the multipliers we follow the same analysis as in proposition \ref{firstorder}. From equations \eqref{DAELd2} we arrive at
\[
\begin{split}
(\lambda_{k+1})_{\alpha}=&-C_{\alpha\beta}(\tilde q_{k+1/2})\left[\frac{1}{2}(v_k+v_{k+1})\frac{\der\mu^{\beta}}{\der q}v_k\right.\\&\left.-\frac{1}{2}\mu^{\beta}(\tilde q_{k+1})M^{-1}\lp\nabla_qV(\tilde q_{k+1})+\nabla_qV(\tilde q_{k})\rp\right]+O(\epsilon).
\end{split}
\]
Taking the Taylor expansion of this expression and comparing to \eqref{LamExp}, it follows directly that $|\lambda_{\alpha}(t_k+\epsilon)-(\lambda_{k+1})_{\alpha}|\simeq O(\epsilon)$. Accordingly, $s=0$ and the claim holds.	
\end{proof}
As in the previous example, using the Ehresmann connection, the equations \eqref{DAELd2a}, \eqref{DAELd2b} and \eqref{DAELd2c} can be decomposed into
\begin{eqnarray*}
\tilde q_{k+1/2}^a&=&\tilde q_k^a+\frac{1}{2}\epsilon v_k^a,\\
\tilde q_{k+1/2}^{\alpha}&=&\tilde q_k^{\alpha}-\frac{1}{2}\epsilon\,A^{\alpha}_a(\tilde q_k) v_k^a,\\
\tilde q_{k+1}^a&=&\tilde q_{k+1/2}^a+\frac{1}{2}\epsilon v_{k+1}^a,\\
\tilde q_{k+1}^{\alpha}&=&\tilde q_{k+1/2}^{\alpha}-\frac{1}{2}\epsilon\,A^{\alpha}_a(\tilde q_{k+1/2}) v_{k+1}^a,\\
v_{k+1}^a&=&v_k^a+\epsilon (\tilde\lambda_{k+1})_{\alpha}\lp M^{ab}A^{\alpha}_b(\tilde q_{k+1/2})+M^{a\alpha}\rp,\\
\end{eqnarray*}
where now, due to  \eqref{lamdaGR},   
\begin{eqnarray*}
(\tilde\lambda_{k+1})_{\alpha}=&-&C_{\alpha\beta}(\tilde q_{k+1/2})[ \frac{1}{2}v^b_k\frac{\der A^{\beta}_b(\tilde q_k)}{\der q^a}v^a_k-\frac{1}{2}v^b_k\frac{\der A^{\beta}_b(\tilde q_k)}{\der q^{\gamma}}A^{\gamma}_a(\tilde q_k)\,v^a_k\\
&+&\frac{1}{2}v^b_{k+1}\frac{\der A^{\beta}_b(\tilde q_{k+1/2})}{\der q^a}v^a_{k+1}-\frac{1}{2}v^b_{k+1}\frac{\der A^{\beta}_b(\tilde q_{k+1/2})}{\der q^{\gamma}}A^{\gamma}_a(\tilde q_{k+1/2})\,v^a_{k+1}\\
&+&\frac{\epsilon}{4}v^a_k\,v^b_k\,v^c_k\frac{\der^2A^{\beta}_a(\tilde q_k)}{\der q^b\der q^c}-\frac{\epsilon}{2}v^a_k\,v^b_k\,v^c_k \frac{\der^2A^{\beta}_a(\tilde q_k)}{\der q^b\der q^{\gamma}}A^{\gamma}_c(\tilde q_k)\\
&+&\frac{\epsilon}{4}v^a_k\,v^b_k\,v^c_k\frac{\der^2A^{\beta}_a(\tilde q_k)}{\der q^{\gamma}\der q^{\rho}}A^{\gamma}_b(\tilde q_k)A^{\rho}_c(\tilde q_k)].
\end{eqnarray*}

\subsection{An elucidating toy model: the nonholonomic particle}

Consider a particle of unit mass evolving in $Q=\R^3$ with Lagrangian
\[
L(x,y,z,\dot x,\dot y,\dot z)=\frac{1}{2}(\dot x^2+\dot y^2+\dot z^2)
\]
and subject to the constraint
\[
\dot z-y\,\dot x=0.
\]
Note that the nonholonomic particle is a simple mechanical system and therefore we can apply all the results developed in the previous subsection. We have a nonholonomic system defined by the annhilation of the one-form $\mu(x,y,z)=(-y,0,1)$. The nonholonomic equations are given by
\begin{equation}\label{NhPequations}
\begin{array}{lcl}
\dot v_x=-y\lambda, & &v_x=\dot x,\\\\
\dot v_y=0, & &v_y=\dot y,\\\\
\dot v_z=\lambda, & &v_z=\dot z,\\\\
       &v_z-y\,v_x=0, & 
\end{array}
\end{equation}
which, after eliminating the Lagrange multiplier using \eqref{lambdax}
\[
\lambda=\frac{v_xv_y}{1+y^2},
\]
lead to 
\begin{equation}\label{NhPequationsSL}
\begin{array}{lcl}
\dot v_x=-\frac{y}{1+y^2}\,v_x\,v_y, & &v_x=\dot x,\\\\
\dot v_y=0, & &v_y=\dot y,\\\\
\dot v_z=\frac{1}{1+y^2}\,v_x\,v_y, & &v_z=\dot z,\\\\
\end{array}
\end{equation}
such that $v_z-y\,v_x=0$. Let $D^{np}\subset T\R^{3}\simeq \R^3\times\R^3$ denote the nonholonomic distribution corresponding to this system. Then the equations in \eqref{NhPequationsSL} define a time-continuous flow $F_t:D^{np}\Flder D^{np}$, i.e. $F_t((q(0),v(0)))=(q(t),v(t))$, where $q(t)=(x(t),y(t),z(t))^T$ and $v(t)=(v_x(t),v_y(t),v_z(t))^T$. We shall consider two examples of the nonholonomic variational integrator \eqref{DLA} with $\phi^{\alpha}_d=-\epsilon\,\mu^{\alpha}\circ\rho$. 
\medskip

First, let us set
\begin{equation}\label{Facil}
\rho(q_k,q_{k+1})=\lp q_{k+1},\frac{q_{k+1}-q_k}{\epsilon}\rp,
\end{equation}
leading to the integrator
\begin{subequations}\label{DLA1NhP}
\begin{align}
\frac{x_{k+1}-2x_k+x_{k-1}}{\epsilon^2}&=-\lambda_{k+1}\,y_k,\label{DLA1NhPa}\\
\frac{y_{k+1}-2y_k+y_{k-1}}{\epsilon^2}&=0,\label{DLA1NhPb}\\
\frac{z_{k+1}-2z_k+z_{k-1}}{\epsilon^2}&=\lambda_{k+1}\,\label{DLA1NhPc}\\
0&=\frac{z_{k+1}-z_k}{\epsilon}-y_{k+1}\,\lp\frac{x_{k+1}-x_k}{\epsilon}\rp.\label{DLA1NhPd}
\end{align}
\end{subequations}
Let us define the finite difference map $\rho(q_k,q_{k+1})\in T_{q_{k+1}}Q$ by setting $v_{k+1}:=v_{q_{k+1}}:=(q_{k+1}-q_k)/\epsilon$ ($q_k=(x_k, y_k, z_k)^T$ and $v_k=(v_{x}^k,v_{y}^k,v_{z}^k)^T$), which yields
\begin{eqnarray*}
&&\frac{v_{x}^{k+1}-v_{x}^k}{\epsilon}=-\lambda_{k+1}\,y_k,\quad\quad\quad\quad\,\,\,\,\,\,\,x_{k+1}=x_k+\epsilon\,v_{x}^{k+1},\nonumber\\\nonumber\\
&&\frac{v_{y}^{k+1}-v_{y}^k}{\epsilon}=0,\,\,\,\,\,\,\,\,\,\,\,\,\,\,\,\,\,\,\,\,\,\,\,\,\,\,\,\,\,\,\,\,\,\,\,\,\,\,\,\,\,\,\,\,\quad\,\,\, y_{k+1}= y_k+\epsilon\,v_{y}^{k+1},\\\nonumber\\
&&\frac{v_{z}^{k+1}-v_{z}^k}{\epsilon}=\lambda_{k+1},\,\,\,\,\,\,\,\,\,\,\,\,\,\,\,\,\,\,\,\,\,\,\,\,\,\,\,\,\,\,\,\,\,\,\,\,\,\,\,\,\,\,\,\,z_{k+1}= z_k+\epsilon\,v_{z}^{k+1},\nonumber\\\nonumber\\
&&\,\,\,\,\,\,\,\,\,\,\,\,\,\,\,\,\,\,\,\,\,\,\,\,\,\,\,\,\,\,\,\,\,\,\,0=v_{z}^{k+1}-y_{k+1}\,v_{x}^{k+1}.\nonumber
\end{eqnarray*}
Notice that, according to \eqref{DLA1NhPd} and the last equation just above, in this case $D_d=D$ and therefore this integrator respects the original constraint {\it without} the need of redefining the nodes. Employing the equation \eqref{lamdaGR}, we obtain the Lagrange multiplier 
\[
\lambda_{k+1}=\frac{v_y^kv_x^k}{1+y_k\,(y_k+\epsilon\,v_y^k)}
\]
and furthermore the integrator
\begin{eqnarray*}
&&\frac{v_{x}^{k+1}-v_{x}^k}{\epsilon}=-y_k\,\frac{v_y^kv_x^k}{1+y_k\,(y_k+\epsilon\,v_y^k)},\quad\quad\quad\quad\,\,\,\,\,\,\,x_{k+1}=x_k+\epsilon\,v_{x}^{k+1},\nonumber\\\nonumber\\
&&\frac{v_{y}^{k+1}-v_{y}^k}{\epsilon}=0,\,\,\,\,\,\,\,\,\,\,\,\,\,\,\,\,\,\,\,\,\,\,\,\,\,\,\,\,\,\,\,\,\,\,\,\,\,\,\,\,\,\,\,\,\quad\quad\quad\quad\quad\quad\quad y_{k+1}= y_k+\epsilon\,v_{y}^{k+1},\\\nonumber\\
&&\frac{v_{z}^{k+1}-v_{z}^k}{\epsilon}=\frac{v_y^kv_x^k}{1+y_k\,(y_k+\epsilon\,v_y^k)},\,\,\,\,\,\,\,\,\,\,\,\,\,\,\,\,\,\,\,\,\,\,\,\,\,\,\,\,\,\,\,\,\,\,\,\,\,\,\,\,\,\,\,\,z_{k+1}= z_k+\epsilon\,v_{z}^{k+1},\nonumber
\end{eqnarray*}
which, according to the results in the previous section, is a $(1,0)$ integrator in the sense of definition \ref{pdisc}. In the following case we focus on a more involved scheme that shows the procedure of redefinition of the nodes.
\medskip

Let us now consider
\begin{equation}\label{NoFacil}
\rho(q_k,q_{k+1})=\lp\frac{q_k+q_{k+1}}{2},\frac{q_{k+1}-q_k}{\epsilon}\rp.
\end{equation}
The nonholonomic variational integrator \eqref{DLA} reads
\begin{subequations}\label{DLA2NhP}
\begin{align}
\frac{x_{k+1}-2x_k+x_{k-1}}{\epsilon^2}&=-\lambda_{k+1}\,y_k,\label{DLA2NhPa}\\
\frac{y_{k+1}-2y_k+y_{k-1}}{\epsilon^2}&=0,\label{DLA2NhPb}\\
\frac{z_{k+1}-2z_k+z_{k-1}}{\epsilon^2}&=\lambda_{k+1}\,\label{DLA2NhPc}\\
0&=\frac{z_{k+1}-z_k}{\epsilon}-\lp\frac{y_{k+1}+y_k}{2}\rp\lp\frac{x_{k+1}-x_k}{\epsilon}\rp.\label{DLA2NhPd}
\end{align}
\end{subequations}
As the discrete Lagrange-d'Alembert principle ensures, these equations generate a discrete nonholonomic flow map $F_{L_d}^{nh}:D_d\Flder D_d$, where in this case $D_d\subset \R^3\times \R^3$.  Let us define the finite difference map $\rho(q_k,q_{k+1})\in T_{\tilde q_{k+1}}Q$, by setting $\tilde q_{k+1}:=(q_{k+1}+q_k)/2$ and $v_{k+1}:=v_{\tilde q_{k+1}}:=(q_{k+1}-q_k)/\epsilon$ ($\tilde q_k=(\tilde x_k,\tilde y_k,\tilde z_k)^T$ and $v_k=(v_{\tilde x}^k,v_{\tilde y}^k,v_{\tilde z}^k)^T$).
Then the integrator given by the scheme \eqref{DLA2NhP} is a particular example for the second case studied in the previous subsection, i.e., it represents a $(2,0)$ consistent discretization (in the sense of definition \ref{pdisc}) of the nonholonomic flow given by the $ODE$ in \eqref{ODExMec}. 
In other words, under these assumptions the equations \eqref{DLA2NhP} can be rewritten as  
\begin{eqnarray}
&&\frac{v_{\tilde x}^{k+1}-v_{\tilde x}^k}{\epsilon}=-\lambda_{k+1}\,\lp \tilde y_k+\frac{\epsilon}{2}v_{\tilde y}^k\rp,\,\,\,\,\,\tilde x_{k+1}=\tilde x_k+\epsilon\,v_{\tilde x}^{k+1},\nonumber\\\nonumber\\
&&\frac{v_{\tilde y}^{k+1}-v_{\tilde y}^k}{\epsilon}=0,\,\,\,\,\,\,\,\,\,\,\,\,\,\,\,\,\,\,\,\,\,\,\,\,\,\,\,\,\,\,\,\,\,\,\,\,\,\,\,\,\,\,\,\,\quad\,\,\, \tilde y_{k+1}=\tilde y_k+\epsilon\,v_{\tilde y}^{k+1},\label{tildenodes0}\\\nonumber\\
&&\frac{v_{\tilde z}^{k+1}-v_{\tilde z}^k}{\epsilon}=\lambda_{k+1},\,\,\,\,\,\,\,\,\,\,\,\,\,\,\,\,\,\,\,\,\,\,\,\,\,\,\,\,\,\,\,\,\,\,\,\,\,\,\,\,\,\,\,\,\tilde z_{k+1}=\tilde z_k+\epsilon\,v_{\tilde z}^{k+1},\nonumber\\\nonumber\\
&&\,\,\,\,\,\,\,\,\,\,\,\,\,\,\,\,\,\,\,\,\,\,\,\,\,\,\,\,\,\,\,\,\,\,\,0=v_{\tilde z}^{k+1}-\tilde y_{k+1}\,v_{\tilde x}^{k+1},\nonumber
\end{eqnarray}
which, by using the formula \eqref{lamdaGR} to determine the Lagrange multiplier 
\[
\lambda_{k+1}=\frac{v_{\tilde x}^k\,v_{\tilde y}^k}{1+(\tilde y_k+\epsilon v_{\tilde y}^k)(\tilde y_k+\frac{\epsilon}{2} v_{\tilde y}^k)},
\]
become
\begin{equation}\label{tildenodes}
\begin{array}{ll}
\frac{v_{\tilde x}^{k+1}-v_{\tilde x}^k}{\epsilon}=-\frac{v_{\tilde x}^k\,v_{\tilde y}^k}{1+(\tilde y_k+\epsilon v_{\tilde y}^k)(\tilde y_k+(\epsilon/2) v_{\tilde y}^k)}\lp \tilde y_k+\frac{\epsilon}{2}v_{\tilde y}^k\rp, & \tilde x_{k+1}=\tilde x_k+\epsilon\,v_{\tilde x}^{k+1},\\\\
\frac{v_{\tilde y}^{k+1}-v_{\tilde y}^k}{\epsilon}=0, &\tilde y_{k+1}=\tilde y_k+\epsilon\,v_{\tilde y}^{k+1},\\\\
\frac{v_{\tilde z}^{k+1}-v_{\tilde z}^k}{\epsilon}=\frac{v_{\tilde x}^k\,v_{\tilde y}^k}{1+(\tilde y_k+\epsilon v_{\tilde y}^k)(\tilde y_k+(\epsilon/2) v_{\tilde y}^k)}\,,& \tilde z_{k+1}=\tilde z_k+\epsilon\,v_{\tilde z}^{k+1},\\
\end{array}
\end{equation}
providing a discrete flow $\tilde F_{L_d}:D_{\tilde q_k}^{np}\Flder D_{\tilde q_{k+1}}^{np}$ such that $\tilde F_{L_d}(\tilde q_k,v_{\tilde q_k})=F_{\epsilon}(q(0),v(0))+O(\epsilon^3)$.

On the other hand, it is easy to see that with respect to the original nodes, this integrator does not preserve the nonholonomic constraint $D^{np}$ and, on top of that, the order of consistency deminishes from second to first order. This assertion can be understood in the following way: if we consider $v_{k+1}:=(q_{k+1}-q_k)/\epsilon$ to be the velocity associated to $q_{k+1}$, then the equations \eqref{DLA2NhP} can be rewritten as
\begin{eqnarray}
&&\frac{v_{x}^{k+1}-v_{x}^k}{\epsilon}=-\lambda_{k+1}\,y_k,\,\,\,\,\,\,\,\,\,\,\,\,\,\,\,\,\,\,\,\,\,\,\,\,\,\,\,\,\,\,\,\,\,\,\,\, x_{k+1}=x_k+\epsilon\,v_{x}^{k+1},\nonumber\\\nonumber\\
&&\frac{v_{y}^{k+1}-v_{y}^k}{\epsilon}=0,\,\,\,\,\,\,\, \,\,\,\,\,\,\,\,\,\,\,\,\,\,\,\,\,\,\,\,\,\,\,\,\,\,\,\,\,\,\,\,\,\,\,\,\,\,\,\,\,\quad\,\,\,y_{k+1}=y_k+\epsilon\,v_{y}^{k+1},\label{original0}\\\nonumber\\
&&\frac{v_{z}^{k+1}-v_{z}^k}{\epsilon}=\lambda_{k+1},\, \,\,\,\,\,\,\,\,\,\,\,\,\,\,\,\,\,\,\,\,\,\,\,\,\,\,\,\,\,\,\,\,\,\,\,\,\,\,\,\,\,\,\,\,\,\,\,z_{k+1}=z_k+\epsilon\,v_{z}^{k+1},\nonumber\\\nonumber\\
&&\,\,\,\,\,\,\,\,\,\,\,\,\,\,\,\,\,\,\,\,\,\,\,\,\,\,0=v_z^{k+1}-y_{k+1}\,v_x^{k+1}+\frac{\epsilon}{2}\,v_y^{k+1}\,v_x^{k+1}.\nonumber
\end{eqnarray}
It is clear that the pairs $(q_{k+1},v_{k+1})$ do not obey the nonholonomic constraint given by  $v_z-y\,v_x=0$ but a deformed one given by 
\begin{equation}\label{deformedCons}
v_z-y\,v_x+\frac{\epsilon}{2}\,v_x\,v_y=0.
\end{equation}
Let $D_{\epsilon}^{np}$ denote the deformed nonholonomic submanifold represented by \eqref{deformedCons} (note that we can not say {\it distribution} since this constraint is no longer linear nor affine). At this point we have two choices for the initial data, namely: either they respect $D^{np}$ or  $D_{\epsilon}^{np}$.
In both cases the order of consistency is deminished from second to first order as compared to the case of the nodes $\tilde q_k$. In the first case, the previous numerical scheme defines a discrete flow map $F_{L_d}:T_{q_k}Q\Flder T_{q_{k+1}}Q$, while in the second case it describes a discrete flow map respecting the deformed constraints, that is $\tilde F_{L_d}^{\epsilon}:\lp D_{\epsilon}^{np}\rp_{q_k}\Flder \lp D_{\epsilon}^{np}\rp_{q_{k+1}}$. Considering the particular expression of the Lagrange multiplier, given by inserting the dynamical variables into the constraints in \eqref{original0}, i.e.
\[
\lambda_{k+1}=\frac{v_y^k\,v_x^k}{1+y_k(y_k+\frac{\epsilon}{2}v_y^k)},
\]
in the second case we get the numerical scheme
\begin{equation}\label{originalnodes}
\begin{array}{lcl}
\frac{v_{x}^{k+1}-v_{x}^k}{\epsilon}=-\frac{y_k}{1+y_k(y_k+(\epsilon/2)v_y^k)}\,v_x^k\,v_y^k,& & x_{k+1}=x_k+\epsilon\,v_{x}^{k+1},\\\\
\frac{v_{y}^{k+1}-v_{y}^k}{\epsilon}=0, & &y_{k+1}=y_k+\epsilon\,v_{y}^{k+1},\\\\
\frac{v_{z}^{k+1}-v_{z}^k}{\epsilon}=\frac{v_y^k\,v_x^k}{1+y_k(y_k+(\epsilon/2)v_y^k)},\, & &z_{k+1}=z_k+\epsilon\,v_{z}^{k+1},\\
\end{array}
\end{equation}
which is first-order consistent with respect to the time-continuous flow $F_t$, i.e. $\tilde F_{L_d}^{\epsilon}(q_k,v_k)=F_{\epsilon}((q(0),v(0)))+O(\epsilon^2)$. Let us define the finite difference map $\tilde\rho$ : $(q_k,q_{k+1})\mapsto(\tilde q_{k+1},v_{\tilde q_{k+1}})$, and the corresponding map $\hat\rho$ : $(q_{k},q_{k+1})\mapsto(q_{k+1},v_{q_{k+1}})$, i.e., 
\begin{subequations}\label{ros}
\begin{align}
\tilde\rho(q_k,q_{k+1})&=\lp\frac{q_k+q_{k+1}}{2},\frac{q_{k+1}-q_k}{\epsilon}\rp,\label{rosa}\\
\hat\rho(q_k,q_{k+1})&=\lp q_{k+1},\frac{q_{k+1}-q_k}{\epsilon}\rp.\label{rosb}
\end{align}
\end{subequations}
To avoid confussion, we would like to point out that with some abuse of notation both $\tilde\rho$ and $\hat\rho$ are finite difference maps associated to the discretization \eqref{NoFacil} corresponding to {\it redefined} and {\it non-redefined} nodes respectively; while, regarding \eqref{Facil}, the finite difference map that we have considered associates the velocity $v_{k+1}$ to $q_{k+1}$ without redefining the nodes.
Taking this into account, the previous discussion can be summarized in the following commutative diagram and subsequent proposition
\[
\xymatrix{
(q_{k},v_{q_k})\ar[r]^{F_{L_d}^{\epsilon}} &(q_{k+1},v_{q_{k+1}})\ar[r]^{F_{L_d}^{\epsilon}} & (q_{k+1},v_{q_{k+1}})\ar[r]^{F_{L_d}^{\epsilon}} &...\\
(q_{k-1},q_k)\ar[r]^{F_{L_d}^{nh}}\ar[u]^{\hat\rho}\ar[d]_{\tilde\rho} &(q_k,q_{k+1})\ar[r]^{F_{L_d}^{nh}}\ar[u]^{\hat\rho}\ar[d]_{\tilde\rho} & (q_{k+1},q_{k+2})\ar[r]^{F_{L_d}^{nh}}\ar[u]^{\hat\rho}\ar[d]_{\tilde\rho} &\ar[u]^{\hat\rho}\ar[d]_{\tilde\rho}...\\
(\tilde q_{k},v_{\tilde q_k})\ar[r]_{\tilde F_{L_d}} &(\tilde q_{k+1},v_{\tilde q_{k+1}})\ar[r]_{\tilde F_{L_d}} & (\tilde q_{k+1},v_{\tilde q_{k+1}})\ar[r]_{\tilde F_{L_d}} &...\\
}
\]
where, we recall, $F_{L_d}^{nh}:D_d\Flder D_d$ is generated  by \eqref{DLA2NhP}, $\tilde F_{L_d}:D_{\tilde q_k}^{np}\Flder D_{\tilde q_{k+1}}^{np}$ by \eqref{tildenodes} and $F_{L_d}^{\epsilon}:\lp D^{np}_{\epsilon}\rp_{q_k}\Flder \lp D^{np}_{\epsilon}\rp_{q_{k+1}}$ by \eqref{originalnodes}.

\begin{proposition}
Consider the nonholonomic particle system and the associated nonholonomic ODE \eqref{NhPequationsSL}. The velocity nonholonomic variational integrator defined by the finite difference map in \eqref{rosa} gives rise to $(2,0)$ consistent discretization, and moreover it is $D^{np}$-preserving. Moreover, the velocity nonholonomic variational integrator defined by the finite difference map \eqref{rosb} gives rise to a $(1,0)$ consistent discretization that does not preserve the original nonholonomic distribution but it does respect the perturbed constraint given by \eqref{deformedCons}.
\end{proposition}
The previous exposition and this proposition emphasize again the importance of the redefinition of the nodes exposed above. Namely, the integrator determined by $\hat\rho$ \eqref{rosb} is {\it first-order}, $(1,0)$ in the nomenclature introduced in definition \ref{pdisc}. On the other hand, the midpoint quadrature $\tilde\rho$ \eqref{rosa} is only {\it second-order} as expected, this is $(2,0)$, after the redefinition of the nodes, which also implies the preservation of the constraints; otherwise the method is first-order in the variables $(q,v)$ and we do not get any improvement as compared to $\hat\rho$ or the simpler $(1,0)$ scheme \eqref{Facil}, where we do not redefine the nodes at all and, yet, we obtain a distribution preserving integrator.

Finally, we would like to make a remark concerning the performance of a preliminary implementation of the numerical schemes introduced in this section (regarding the nonholonomic particle). They provided the results expected according to the general theory presented in this paper: In all cases, the original constraint turned out to be preserved up to small bounded oscillatory variations. We will provide more details of the numerics associated to the theory developed in this paper in forthcoming papers.  

\begin{remark}\label{deformation}
{\rm
The example above draws our attention to the case of integrators that {\it do not} preserve the nonholonomic constraints (which is the usual case) but a deformation of them. At the light of the example of the nonholonomic particle and the perturbed constraints \eqref{deformedCons}, we shall consider the following deformation:
\begin{equation}\label{DefCons}
\mu^{\alpha}_i(q)\,v^i+\delta\,g^{\alpha}(q,v)=0,
\end{equation}
where we consider $\delta>0$ to be a small parameter and $g^{\alpha}:TQ\Flder\R$ is a generic smooth function, where $Q=\R^n$, not necessarily linear in the velocities (thus, we cannot speak anymore about a distribution but about a {\it submanifold}).  According to theorem \ref{Subma}, we need further assumptions on $g^{\alpha}$ such that \eqref{DefCons} defines a regular submanifold  $D_{\delta}\subset T\R^n$, namely
\[
\mbox{rank}\lp\frac{\der\mu^{\alpha}_i}{\der q^j}v^i+\delta\,\frac{\der g^{\alpha}}{\der q^j}\,,\,\mu^{\alpha}_i+\delta\,\frac{\der g^{\alpha}}{\der v^i}\rp=m.
\]
It is enough to assume that rank $\lp\mu^{\alpha}_i+\delta\,\frac{\der g^{\alpha}}{\der v^i}\rp=m$, an assumption that we are making in the following (note that this is locally true for $\delta$ small enough). Con\-si\-de\-ring equations \eqref{NhDAE3} and theorem \ref{EUDAETh} we shall define the nonholonomic dynamics on $D_{\delta}$ by means of the following set of differential algebraic equations:
\begin{subequations}\label{DefDAE}
\begin{align}
\dot y&=f(y)+\lambda_{\alpha}\lp\begin{array}{c}0\\m^{-1}\nabla_v(\varphi^{\alpha}(y)+\delta\,g^{\alpha}(y))\end{array}\rp,\label{DefDAEa}\\
0&=\varphi^{\alpha}(y)+\delta\,g^{\alpha}(y),\label{DefDAEb}
\end{align}
\end{subequations}
which are a $\delta-$perturbation of the original nonholonomic DAE \eqref{NhDAE3}. Note again that theorem \ref{EUDAETh} allows more freedom on the function $g^{\alpha}$, more concretely, it might be non-autonomous (although we are keeping the autonomous setting for now). Moreover, we can employ analogous arguments than in lemma \ref{lemmaC} to prove that the $m\times m$ matrix
\begin{equation}\label{Ceps}
(C^{\alpha\beta}_{\delta}(q,v))=\lp(\mu^{\alpha}_i+\delta\,\frac{\der g^{\alpha}}{\der v^i})\,m^{ij}\,(\mu^{\beta}_j+\delta\,\frac{\der g^{\beta}}{\der v^j})\rp
\end{equation}
is invertible for small $\delta$. Consequently, equations \eqref{DefDAE} induce the following ODE evolving on $D_{\delta}$
\begin{equation}\label{heps}
\dot x=f(x)+\lambda_{\alpha}^{\delta}(x)\left(\begin{array}{c}0\\m^{-1}\nabla_v(\varphi^{\alpha}(x)+\delta\,g^{\alpha}(x))\end{array}\right),
\end{equation}
where 
\begin{multline}
\lambda_{\alpha}^{\delta}(x)=-\lp C_{\delta}(x)\rp_{\alpha\beta}\left(\bra\nabla_q(\varphi^{\beta}(x)+\delta\,g^{\beta}(x)),f_q(x)\ket\right.\\
\left.+\bra\nabla_v(\varphi^{\beta}(x)+\delta\,g^{\beta}(x)),f_v(x)\ket\right),
\end{multline}
and the constraints $\varphi^{\alpha}(x)+\delta\,g^{\alpha}(x)=0$ must be obeyed. Introducing local coordinates $\xi_{\delta}\in\R^{2n-m}$ and the mapping $\psi_{\delta}:\R^{2n-m}\Flder D_{\delta}\subset\R^{2n}$ such that for $x\in D_{\delta}\subset\R^{2n}$, $x=\psi_{\delta}(\xi_{\delta})$, the previous equation  induces an ODE on $\R^{2n-m}$, i.e.
\begin{equation}\label{ODExiEps}
\dot\xi_{\delta}=\lp\nabla_{\xi}\psi_{\delta}\rp^{-1}(\xi_{\delta})\,h^{\delta}\lp\psi_{\delta}(\xi_{\delta})\rp,
\end{equation}
where the vector field $h^{\delta}(x)$ is given by the right hand side of \eqref{heps}. Also, we can employ the
Ehresmann connection (which is globally well-defined), to choose
suitable bundle coordinates such that the result in proposition \ref{InterpolationD} can be applied to define a curve $c_{\delta}(t)\subset D_{\delta}.$ Furthermore, the analysis presented in corollary \ref{Coro} and remark \ref{LargeRemark} can be applied in the case $D_{\delta}$, leading to analogous conclusions.
}
\end{remark}

\section{Conclusions}

We have proved that any $D-$preserving integrator of the nonholonomic ODE may be understood as a non-autonomous perturbation of this ODE itself. This result is based on the construction of the nonholonomic ODE from the original nonholonomic DAE, a  construction which strongly depends on the structure of the linear constraints (inherent in our nonholonomic setting) and the regularity conditions which we impose. It is also studied what kind of a 
perturbation a non-autonomous perturbation of the nonholonomic DAE may cause for the nonholonomic ODE. The matching of both perturbations and their relation to the continuous Lagrange-d'Alembert principle is left as a topic of further analysis in future works. 

Moreover, from the discretization of the Lagrange-d'Alembert principle we construct $D-$preserving variational integrators. We give precise conditions under which this property is ensured. It turns out that this property requires a redefinition of the original discretization nodes. We illustrate our procedure by means of two particular velocity nonholonomic integrators whose consistency properties are carefully studied in the case of simple mechanical systems.

We employ the developed techniques to treat the example of the nonholonomic particle in order to show how the discretizations may induce a perturbation of the original dynamics, particularly of the nonholonomic constraints. Our conclusions regarding this example can be summarized as follows: when we construct the velocity nonholonomic integrator redefining the nodes to preserve $D$ we obtain the maximum order of consistency achievable; when we preserve the original nodes  and perturb $D$ to some $D_{\epsilon}$ we loose one order of consistency; 
when we preserve the nodes and do not deform the distribution we loose one order of consistency; when we keep the original nodes and deform the constraints in order to preserve the original distribution in the numerical scheme we loose one order of consistency as well. This is an interesting phenomenon which we are going to consider in more detail and in a more general situation in future works.

Finally, in the light of the possible deformation of the distribution generated by discretizations, we extended the theory developed before to allow even nonlinear deformations of the constraints (consequently we are not dealing with a {\it distribution} anymore but with a smooth constraint {\it submanifold}).

\end{document}